\RequirePackage[l2tabu, orthodox]{nag} 
\documentclass[10pt]{amsart}

\usepackage[utf8]{inputenc}
\usepackage[british]{babel}
\usepackage{amsfonts,amssymb,amsmath,amsthm}
\usepackage{microtype}
\usepackage{lmodern}
\usepackage[final]{showkeys} 
\usepackage{url}
\usepackage{enumitem}
  \setitemize{leftmargin=*}   
  \setenumerate{leftmargin=*} 

\usepackage[
  hidelinks
]{hyperref}

\newcommand{\R}{\mathbb{R}} 
\newcommand{\N}{\mathbb{N}} 
\DeclareMathOperator{\dive}{\operatorname{div}} 
\DeclareMathOperator*{\esssup}{ess\,sup} 
\newcommand{\Hm}{\mathcal{H}} 
\newcommand{\stc}{{\sigma_\mathrm{st}}} 

\newtheorem{theorem}{Theorem}[section]
\newtheorem{lemma}[theorem]{Lemma}
\newtheorem{definition}[theorem]{Definition}
\newtheorem{assumptions}[theorem]{Assumptions}
\newtheorem{proposition}[theorem]{Proposition}
\newtheorem{corollary}[theorem]{Corollary}
\newtheorem{remark}[theorem]{Remark}

\newcounter{countproofstep} 
\newcommand{\proofstep}{
  \addtocounter{countproofstep}{1}%
  \par
  \textit{Step~\thecountproofstep. }%
  \nopagebreak%
}

\numberwithin{equation}{section}%
\begin{document}

%
%
\title[Pressure Reconstruction for the 2PNS Equations]{Pressure Reconstruction for Weak Solutions of the Two-Phase Incompressible Navier--Stokes Equations with Surface Tension}
\date{\today}
\author{Helmut Abels}
\address{Fakult\"at f\"ur Mathematik, Universit\"at Regensburg\\Universit\"atsstr.\ 31, 93053 Regensburg, Germany}
\email{helmut.abels@ur.de}
\author{Johannes Daube}
\address{Abteilung f\"ur Angewandte Mathematik, Universit\"at Freiburg\\Hermann-Herder-Str.\ 10, 79104 Freiburg, Germany}
\email{johannes.daube@mathematik.uni-freiburg.de}
\author{Christiane Kraus}
\address{Weierstra{\ss}-Institut\\Mohrenstr.\ 39, 10117 Berlin, Germany}
\email{kraus@wias-berlin.de}
%
%

%
%
\begin{abstract}
  For the two-phase incompressible Navier--Stokes equations with surface tension, we derive an appropriate weak formulation incorporating a variational formulation using divergence-free test functions.
  We prove a consistency result to justify our definition and, under reasonable regularity assumptions, we reconstruct the pressure function from the weak formulation.
\end{abstract}
\maketitle
%
%
\section{Introduction}\label{sec:introduction}
%
%
We consider a two-phase flow of two incompressible Newtonian fluids.
The isothermal flow in a bounded domain $\Omega \subset \R^n$, $n=2,3$, and on a finite time interval $[0,T]$ is described in Eulerian coordinates by a velocity field $v\colon \Omega \times [0,T] \rightarrow \R^n$ and a scalar pressure function $p\colon \Omega \times [0,T]\rightarrow \R$.
For each time $t \in [0,T]$, a hypersurface $\Gamma(t)$ separates $\Omega$ into two disjoint subsets $\Omega^-(t)$ and $\Omega^+(t)$ of $\Omega$, i.e., we have $\Omega = \Omega^-(t) \cup \Gamma(t) \cup \Omega^+(t)$ and $\Gamma(t) = \partial \Omega^-(t) \cap \Omega$.
The regions $\Omega^-(t)$ and $ \Omega^+(t)$ are referred to as bulk phases, and correspond to different phases of the fluid.
Physically they are characterised by (constant) densities $0 < \beta_1 \leq \beta_2$ and corresponding viscosities $\mu(\beta_i) > 0$, $i=1,2$.
For convenience, throughout this paper we will require that the interface is compactly contained in the fluid domain, that is, 
$\Gamma(t) \subset\subset \Omega$.
In particular, the interface does not intersect the domain boundary, i.e., $\Gamma(t) \cap \partial \Omega = \emptyset$.
This, in turn, means that $\Omega^-(t) \subset\subset \Omega$ and $\Gamma(t) = \partial \Omega^-(t) = \partial \Omega^-(t) \cap\partial \Omega^+(t)$.
\par
Assuming the interface $\Gamma$ to be sufficiently regular, and the velocity $v$ and the pressure $p$ to be sufficiently smooth functions on $\Omega\setminus\Gamma(t) =  \Omega^-(t) \cup \Omega^+(t)$, such that the one-sided limits on $\Gamma(t)$ from $\Omega^\pm(t)$ exist, the flow is described by the following free-boundary problem
\begin{align}
  \beta_1 \partial_t v + \beta_1 (v \cdot \nabla) v - \mu(\beta_1) \Delta v + \nabla p &= 0 &&\ \text{in} \ \Omega^-(t),\label{eq:sim:bulk1}\\
  \beta_2 \partial_t v + \beta_2 (v \cdot \nabla) v - \mu(\beta_2) \Delta v + \nabla p &= 0 &&\ \text{in} \ \Omega^+(t),\label{eq:sim:bulk2}\\
  \dive(v)&= 0 &&\ \text{in} \ \Omega \setminus \Gamma(t),\label{eq:sim:divergence_free}\\ 
  [v]&= 0 &&\ \text{on} \ \Gamma(t),\label{eq:sim:velocity_jump}\\ 
  V &= v \cdot \nu^- &&\ \text{on} \ \Gamma(t), \label{eq:sim:v_normal_velocity}\\    
  \left[ T \right] \nu^- &= - 2 \stc \kappa \nu^- &&\ \text{on} \ \Gamma(t), \label{eq:sim:pressure_jump}\\
  v(\cdot,t) &= 0 &&\ \text{on} \ \partial \Omega, \label{eq:sim:bc}\\
  v(\cdot,0) &= v^{(i)} &&\ \text{in} \ \Omega \label{eq:sim:ic}
\end{align}
for every $t\in [0,T]$.
The initial phases $\Omega^-(0) = \Omega^{-,(i)} \subset\subset \Omega$, $\Omega^+(0) = \Omega \setminus \overline{\Omega^{-,(i)}}$ and the initial position $\Gamma^{(i)} = \partial ( \Omega^{-,(i)} )$ of the interface, as well as the initial velocity $v^{(i)} \colon \Omega \rightarrow \R^n$, are given.
The unknowns are the velocity $v( \cdot, t ) \colon \Omega \setminus \Gamma (t) \rightarrow \R^n$, the pressure $p( \cdot, t ) \colon \Omega \setminus \Gamma (t) \rightarrow \R$ and the interface (free-boundary) $\Gamma (t)$.
Here and in the sequel, $[\, \cdot \,]$ stands for the jump across the interface $\Gamma(t)$ in the direction of the exterior unit-normal field $\nu^-(\cdot,t)$ of $\partial \Omega^-(t)$.
For a given quantity $f$ and $x\in \Gamma(t)$, this is, explicitly,
\begin{equation*}
  [f](x,t) = \lim_{\xi \searrow 0} \big(f(x+\xi \nu^-(x,t),t) - f(x-\xi \nu^-(x,t),t) \big).
\end{equation*}
By $V = V(\cdot,t)$ and $\kappa = \kappa(\cdot,t)$, we denote the normal velocity and the mean curvature of $\Gamma(t)$, for fixed $t$, both taken with respect to $\nu^-(\cdot,t)$.
Moreover, in \eqref{eq:sim:pressure_jump}, $\stc > 0$ denotes the surface-tension constant, and the stress tensor $T = T( v, p )$ is defined by
\begin{equation*}
  T(v(t),p(t))
  = 
  \begin{cases}
    2\mu(\beta_1)Dv(t) - p(t) I \ & \text{in} \ \Omega^-(t),\\
    2\mu(\beta_2)Dv(t) - p(t) I \ & \text{in} \ \Omega^+(t).
  \end{cases}
\end{equation*}
The partial differential equations \eqref{eq:sim:bulk1}--\eqref{eq:sim:divergence_free} are the incompressible Navier--Stokes equations.
Equations~\eqref{eq:sim:bulk1} and~\eqref{eq:sim:bulk2} model the conservation of linear momentum and the incompressibility condition~\eqref{eq:sim:divergence_free} corresponds to conservation of mass in each bulk phase.
These partial differential equations in the bulk phases are coupled by the interface conditions \eqref{eq:sim:velocity_jump}--\eqref{eq:sim:pressure_jump}: the velocity field is continuous across the interface $\Gamma(t)$ by \eqref{eq:sim:velocity_jump}.
Due to \eqref{eq:sim:v_normal_velocity}, the interface $\Gamma(t)$ is transported purely by the bulk fluid flow.
The interface condition \eqref{eq:sim:pressure_jump} is (a dynamic version of) the Young--Laplace law relating the jump of the normal stress $\left[ T \right] \nu^-$ to the mean curvature $\kappa$.
The velocity boundary condition \eqref{eq:sim:bc} is the no-slip condition at the boundary $\partial \Omega$ of the fluid domain $\Omega$.
With \eqref{eq:sim:ic}, we prescribe initial values $v^{(i)}\colon \Omega\rightarrow \R^n$ for the velocity.
\par
The question of (unique) solvability of the free-boundary problem \eqref{eq:sim:bulk1}--\eqref{eq:sim:ic} and related systems has been studied by many authors: in the framework of H\"older spaces, Denisova and Solonnikov first studied the corresponding two-phase Stokes problem \cite{denisova_solonnikov_stokes}.
Later they proved well-posedness of \eqref{eq:sim:bulk1}--\eqref{eq:sim:ic} for appropriate initial data \cite{denisova_solonnikov_ns}.
Existence results in the context of maximal $L^r$-regularity (so-called strong solutions), which are even real analytic for positive times, are due to Pr\"uss and Simonett \cite{pruess_simonett} and K\"ohne, Pr\"uss and Wilke \cite{koehne_pruess_wilke}, and in a varifold context due to Plotnikov \cite{plotnikov_nonnewtonian, plotnikov_comp_stokes}.
In general, the existence of weak solutions to \eqref{eq:sim:bulk1}--\eqref{eq:sim:ic} is an open problem, cf.\ \cite[Section~2.2]{abels_garcke}.
\par
This paper summarizes the result of \cite[Chapter~4]{daube} and is organised as follows:
In \textbf{Section~\ref{sec:notation}} we will introduce our notation and provide some preliminary results.
In \textbf{Section~\ref{sec:weak_solution}} we will derive a weak notion of solutions which uses divergence-free test functions.
This will lead to a weak formulation that does not incorporate the pressure function.
\par
In the remainder of the paper we shall justify our approach and reconstruct a pressure function from the weak formulation:
In \textbf{Section~\ref{sec:sobolev_time_dep}} we shall provide the functional-analytic background and introduce Sobolev spaces on time-dependent domains.
In \textbf{Section~\ref{sec:sim:consistency}}, under reasonable regularity assumptions, we will reconstruct the pressure function from the weak formulation.
%
%
%
\section{Notation and Preliminaries}\label{sec:notation}
%
%
%
Let $U\subset \R^d$, $d\in \N$, be open.
The space of smooth and compactly supported functions in $U$ is denoted by $C_{0}^\infty( U )$ and $C_{0,\sigma}^\infty( U )$ is the subspace of $C_{0}^\infty( U )$ of divergence-free functions.
Moreover, for $Q \subset \R^d$, we define
\begin{equation*}
  C_{(0)}^\infty (Q) = \{ u \colon Q \rightarrow \R \, : \, u = \left. U \right|_Q, \ U \in  C_{0}^\infty(\R^d), \ \mathrm{supp}( u ) \subset Q \}.
\end{equation*}
For a Banach space X, its dual is designated by $X^\ast$.
For a measurable set $M \subset \R^d$ and $r \in [1,\infty]$, $L^r(M)$ and $L^r(M;X)$ denote the standard Lebesgue spaces of scalar and $X$-valued functions, respectively.
If $M=(a,b)$, we simply write $L^r(a,b;X)$.
$W^{k,r}( U )$ is the Sobolev space of order $k\in \N$ and integrability exponent $r$.
By $W^{k,r}_0( U )$, we denote the closure of $C^\infty_0( U )$ in $W^{k,r}( U )$, and we set $H^k( U ) = W^{k,2}( U )$ and $H^k_0( U ) = W^{k,2}_0( U )$.
The corresponding dual spaces we abbreviate as $W^{-1,q}( U ) = ( W^{1,q^\prime}_0( U ) )^\ast$, where $q^\prime = \tfrac{q}{q-1}$, and $H^{-1}( U ) = W^{-1,2}( U )$. 
Furthermore, $L^2_\sigma( U )$ and $H^1_{0,\sigma}( U )$ denote the closure of $C^\infty_{0,\sigma}( U )$ in $L^2( U )$ and $H^1( U )$, respectively.
For $r \in [1,\infty)$, we define
\begin{equation*}
  L^r_{\sigma}(U)
  = \overline{C_{0,\sigma}^\infty(U)}^{\| \cdot \|_{L^r(U)^d }}
  \ \text{and} \
  W^{1,r}_{0,\sigma}(U)
  = \overline{C_{0,\sigma}^\infty(U)}^{\| \cdot \|_{W^{1,r}(U)^d }}.
\end{equation*}
Furthermore, for $r = 2$, we use the notation $H^1_{0,\sigma}(U) = W^{1,2}_{0,\sigma}(U)$.
The space $W^{1,r}_{0,\sigma}(U)$ has the following useful characterisation; see \cite[Lemma~II.2.2.3]{sohr}.
\begin{lemma}[Characterisation of $W^{1,r}_{0,\sigma}(\Omega)$]\label{lem:sigma_characterisations}
  For $d\geq 2$ and $r \in (1,\infty)$, let $\Omega \subset \R^d$ be a bounded domain with Lipschitz boundary.
  Then there holds
  \begin{equation*}
    W^{1,r}_{0,\sigma}(\Omega) = \big\{ u \in W^{1,r}_0(\Omega)^d \,:\, \dive(u) = 0 \big\}.
  \end{equation*}
\end{lemma}
It is convenient to introduce the spaces
\begin{equation*}
  C_0^\infty( (0,T) ; C_{0,\sigma}^\infty(\Omega) )
  = \big\{ \psi \in C_0^\infty(\Omega \times (0,T) )^d \,:\, \dive( \psi ) = 0 \big\}
\end{equation*}
and
\begin{equation*}
  C_0^\infty( [0,T) ; C_{0,\sigma}^\infty(\Omega) )
  = \big\{ \left. \psi \right|_{\Omega \times [0,T)} \,:\,  \psi \in C_0^\infty(\Omega \times (-1,T) )^d,\ \dive( \psi ) = 0 \big\}.
\end{equation*}
%
\subsection{Functions of Bounded Variation and Sets of Finite Perimeter}
%
For $N\in\N$ and a finite $\R^N$-valued Radon measure $\mu$ and a Borel set $E \subset U$, the total-variation measure of $E$ is defined by
\begin{equation*}
  \left| \mu \right|(E)
  = \sup \sum_{m=1}^\infty { \left| \mu (E_m) \right| },
\end{equation*}
where the supremum is taken over all pairwise disjoint partitions $(E_m)_{m\in \N} \subset X$ of measurable sets $E_m$, $m\in \N$, such that $E = \bigcup_{m=1}^\infty E_m$.
A function $u\in L^1(U)$ is said to be of bounded variation if its distributional gradient $\nabla u$ is a finite $\R^d$-valued Radon measure.
The set of all functions of bounded variation is denoted by $BV(U)$, and the set $BV(U, M)$ contains all functions $u \in BV(U)$, such that $u \in M$ for a.e.\ $x\in U$.
A measurable set $E \subset U$ has finite perimeter in $U$ if its characteristic function $\chi_E$ belongs to $BV(U)$.
By the structure theorem of sets of finite perimeter, there holds $\left|\nabla \chi_E\right|(U) = \Hm^{d-1}( U \cap \partial^\ast E )$, where $\Hm^{d-1}$ is the $(d-1)$-dimensional Hausdorff measure and $\partial^\ast E$ is the so-called reduced boundary of $E$, and, moreover, for all $\psi \in C_0^\infty( U )^d$,
\begin{equation*}
  \int_E \dive( \psi ) \,\mathrm{d} x
  = \int_{\partial^\ast E} \psi \cdot \nu_E  \,\mathrm{d}  \Hm^{d-1}( x ),
\end{equation*}
where $\nu_E (x) = - \lim_{\delta \searrow 0} \tfrac{\nabla \chi_E (B_\delta(x))}{\left| \nabla \chi_E \right|(B_\delta(x))}$ is the generalized outer unit normal; cf.\ e.g.\ \cite[Theorem~3.36]{ambrosio}.
Note that, if $E$ has $C^1$-boundary, then $\partial^\ast E = \partial E$ and $\nu_E$ coincides with the usual outer unit normal.
%
\subsection{Hypersurfaces}\label{sec:hypersurfaces}
%
We briefly recall some facts from differential geometry.
For a fuller treatment, cf.\ for instance \cite[Section~2]{deckelnick}, \cite[Section~16.1]{gilbarg} and \cite{masato}.
We call $\Gamma \subset \R^d$, $d \geq 2$, a $C^k$-hypersurface, $k\in \N$, if, for each $x_0 \in \Gamma$, there exist an open neighbourhood $U\subset \R^d$ of $x_0$ and a function $u \in C^k(U)$ with
\begin{equation*}
  U\cap \Gamma = \{ x\in U \,: \, u(x) = 0\}
  \ \text{and} \
  \nabla u(x) \neq 0 \ \text{for all} \  x \in U \cap \Gamma.
\end{equation*}
In the case $k=2$, we briefly call $\Gamma$ a hypersurface.  
For a hypersurface $\Gamma$, the space $C^1(\Gamma)$ consists of all functions $f\colon \Gamma \rightarrow \R$ such that there exist a neighbourhood $U \subset \R^d$ of $\Gamma$ and a function $g\in C^1( U )$ with $f = \left. g \right|_U$.
The tangent space $T_x \Gamma$ of a hypersurface $\Gamma \subset \R^d$, at a point $x\in \Gamma$, is defined by
\begin{equation*}
  T_x \Gamma = \{ \tau \in \R^d \,:\, \tau \cdot \nabla u(x) = 0 \}.
\end{equation*}
A hypersurface $\Gamma$ is called oriented if there exists a function $\nu \in C^1(\Gamma)^d$ such that, for all $x\in \Gamma$, there holds $\left| \nu (x) \right| = 1$ and $\nu (x) \perp T_x \Gamma$, i.e., $\nu (x) \cdot \tau =0$ for any $\tau \in T_x \Gamma$.
The function $\nu$ is called unit-normal field (or, briefly, normal).
On an oriented hypersurface $\Gamma \subset \R^d$ with unit-normal field $\nu$, for $f\in C^1( \Gamma )$, the tangential gradient $\nabla_\Gamma f \colon \Gamma \rightarrow \R^d$ is defined as
\begin{equation*}
  \nabla_\Gamma f = (\delta_1 f,\dots,\delta_d f) = \left. \big( \nabla f  - ( \nabla f \cdot \nu ) \nu \big) \right|_\Gamma.
\end{equation*}
For $u \in C^1( \Gamma )^d$, the tangential divergence $\dive_\Gamma (u) \colon \Gamma \rightarrow \R$ is defined as
\begin{equation*}
  \dive_\Gamma (u) = \sum_{i=1}^{d} \delta_i u_i.
\end{equation*}

\begin{proposition}\label{prop:properties_curvarute_matrix}
  For an oriented hypersurface $\Gamma \subset \R^d$ with unit-normal field $\nu$, define $\mathcal{K} = (\mathcal{K}_{ij})_{i,j=1,\dots,d}\colon \Gamma \rightarrow \R^{d\times d}$ by
  \begin{equation}\label{eq:def_curvature_matrix}
    \mathcal{K}_{ij}(x) = - \delta_i \nu_j(x)
  \end{equation}
  for $i,j=1,\dots,d$ and $x\in \Gamma$.
  Then, for every $x\in \Gamma$, the matrix $\mathcal{K}(x)$ is symmetric and $\nu(x)$ is an eigenvector of $\mathcal{K}(x)$ with corresponding eigenvalue $0$.
\end{proposition}

\begin{proof}
  See~\cite[Section~2.3]{deckelnick} or \cite[Theorem~2.10]{masato}.
\end{proof}

The foregoing proposition allows one to define the mean curvature of an oriented hypersurface.

\begin{definition}[Mean curvature]\label{def:mean_curvature}
  For an oriented hypersurface $\Gamma \subset \R^d$ with unit-normal field $\nu$, let $x\in \Gamma$ and let $\mathcal{K}$ be defined as in \eqref{eq:def_curvature_matrix}.
  \begin{enumerate}
    \item
      The principal curvatures of $\Gamma$ in $x$ are the eigenvalues $\kappa_1(x), \dots, \kappa_{d-1}(x)$ of $\mathcal{K}(x)$ belonging to eigenvectors orthogonal to $\nu(x)$. 
    
    \item
    The mean curvature $\kappa \colon \Gamma \rightarrow \R$ is the trace of $\mathcal{K}$, i.e.,
      \begin{equation*}
        \kappa
        = \sum_{i=1}^{d} \mathcal{K}_{ii}
        = \sum_{i=1}^{d-1} \kappa_{i}.
      \end{equation*}
  \end{enumerate}
\end{definition}

Note that, in view of the above definitions, there holds $\kappa = - \dive_\Gamma(\nu)$.
Moreover, for $f \in C^1( \Gamma )$ and $i=1,\dots,d$, there holds the integration-by-parts formula
\begin{equation}\label{eq:int_parts_surface}
  \int_{\Gamma} \delta_i f \,\mathrm{d} \Hm^{d-1}(x) = - \int_{\Gamma} f \kappa \nu_i \,\mathrm{d} \Hm^{d-1}(x);
\end{equation}
see \cite[Lemma~16.1]{gilbarg}.

For the treatment of time-dependent interfaces, we need the notion of evolving hypersurface, and have to define its normal velocity.

\begin{definition}[Evolving hypersurfaces]\label{def:evolving_hypersurfaces}
  Let $I \subset \R$ be an interval.
  For a family $(\Gamma(t))_{t\in I} \subset \R^d$ of oriented hypersurfaces, define
  \begin{equation}\label{eq:spacetime_Gamma}
    \Gamma = \bigcup_{t\in I} \big( \Gamma(t) \times \{ t \} \big).
  \end{equation}
  \begin{enumerate}
    \item
      $(\Gamma(t))_{t\in I}$ is called a $C^{2,1}$-family of evolving oriented hypersurfaces, or, briefly, a family of evolving hypersurfaces, if $\Gamma$ is a $C^1$-hypersurface in $\R^{d+1}$ and there exists a function $\nu \in C^1( \Gamma )^d$ such that $\Gamma(t)$ is oriented by $\nu(\cdot,t)$ for every $t \in I$.
 
    \item
      The normal velocity $V \in C^0( \Gamma )$ of $(\Gamma(t))_{t\in I}$ at a point $(x_0,t_0) \in \Gamma$ is given by
      \begin{equation*}
        V(x_0,t_0) = \eta^\prime(t_0) \cdot \nu(x_0,t_0),
      \end{equation*}
      where $\eta \in C^1( I_0 )^d$, for some subinterval $I_0 \subset I$ with $t_0 \in I_0$, such that
      $\eta(t_0) = x_0$ and $\eta(t) \in \Gamma(t)$ for all $t \in I_0$.
\end{enumerate}
\end{definition}

\begin{remark}
  The definition of the normal velocity $V$ does not depend on the choice of the function $\eta$.
  Moreover, for any $t\in I$, there holds $V( \cdot, t) \in C^1( \Gamma(t) )$; see~\cite[Theorem~5.5]{masato}.
\end{remark}

Finally, we provide some transport identities for integrals, which allow one to calculate time derivatives of integrals over time-dependent domains and hypersurfaces.

\begin{theorem}[Transport theorem]\label{thm:transport_theorem}
  For some interval $I \subset \R$, let $(\Gamma(t))_{t\in I}$ be a family of evolving hypersurfaces in the sense of Definition \ref{def:evolving_hypersurfaces}.
  In addition, for every $t\in I$, assume that $\Gamma(t) = \partial\Omega(t)$ for some open, bounded set $\Omega(t) \subset \R^d$.
  Denote by $\nu = \nu(t)$ the unit-normal field of $\Gamma(t)$ pointing outward to $\Omega(t)$, by $\kappa = \kappa(t)$ the mean curvature of $\Gamma(t)$ and by $V=V(t)$ the normal velocity of $(\Gamma(t))_{t\in I}$, respectively, with respect to $\nu (t)$.
  \begin{enumerate}
    \item 
  If $U\subset\R^{d+1}$ is an open set such that
  \begin{equation*}
    \bigcup_{t\in I} \Big( \overline{ \Omega(t) } \times \{t\} \Big) \subset U,
  \end{equation*}
  then, for every $f\in C^1(U)$, there holds
  \begin{equation*}
    \frac{\mathrm d}{\mathrm d t} \int_{\Omega(t)} f \,\mathrm{d} x = \int_{\Omega(t)} \partial_t f \,\mathrm{d} x + \int_{\Gamma(t)} f V  \,\mathrm{d} \Hm^{d-1}(x).
  \end{equation*}
 
 \item
  Let $\Gamma$ be as in \eqref{eq:spacetime_Gamma}.
  If $f\in C^1( \Gamma )$, then, for ant $t\in I$, there holds
  \begin{equation*}
  \begin{split}
    &\frac{\mathrm d}{\mathrm d t} \int_{\Gamma(t)} f \,\mathrm{d} \Hm^{d-1}(x)\\
    ={}&  \int_{\Gamma(t)} \partial_t f \,\mathrm{d} \Hm^{d-1}(x) - \int_{\Gamma(t)} f \kappa V  \,\mathrm{d} \Hm^{d-1}(x)
    + \int_{\Gamma(t)} ( \nabla f \cdot \nu ) V \,\mathrm{d} \Hm^{d-1}(x).
  \end{split}
  \end{equation*}
  In particular,
  \begin{equation*}
    \frac{\mathrm d}{\mathrm d t} \Hm^{d-1}( \Gamma(t) )
    = - \int_{\Gamma(t)} \kappa V  \,\mathrm{d} \Hm^{d-1}(x).
  \end{equation*}
\end{enumerate}
\end{theorem}
  
\begin{proof}
  See \cite[Appendix]{deckelnick} or \cite[Theorems~6.1 and 6.4]{masato}.
\end{proof}  
%
%
\section{The Notion of Weak Solutions}\label{sec:weak_solution}
%
%
The free-boundary problem \eqref{eq:sim:bulk1}--\eqref{eq:sim:ic} incorporates two disjoint subregions $\Omega^-(t)$ and $\Omega^+(t)$ of the domain $\Omega$, where the fluid is of constant density $\beta_1$ and $\beta_2$, respectively.
This means that the associated density function is given by
\begin{equation}\label{eq:def_rho}
  \rho(t)
  = \beta_1 \chi_{\Omega^-(t)} + \beta_2 \chi_{\Omega^+(t)}
   \ \text{in} \ \Omega.
\end{equation}  
Note that $\rho(t) = ( \beta_1 - \beta_2 ) \chi_{\Omega^-(t)} + \beta_2$ in $\Omega \setminus \Gamma(t)$.
Moreover, the nature of $\rho(t)$ is encoded in the characteristic function
\begin{equation}\label{eq:def_chi}
  \chi(t)
  = \chi_{\Omega^-(t)}
  = \frac{\rho(t) - \beta_2}{\beta_1 - \beta_2},
\end{equation}
and vice versa.
In many situations, it is convenient to use that \eqref{eq:sim:bulk1} and \eqref{eq:sim:bulk2} are equivalent to
\begin{equation}\label{eq:sim:assos_density_bulk}
  \rho \partial_t v + \rho (v \cdot \nabla) v  - 2 \mu( \rho ) \dive( D v ) + \nabla p = 0 \ \text{in} \ \Omega \setminus \Gamma(t).
\end{equation}

To motivate a weak formulation, we consider sufficiently smooth solution triplets $(v,p, \Gamma)$ of \eqref{eq:sim:bulk1}--\eqref{eq:sim:ic}; see  Assumptions~\ref{ass:sim:smoothness} below.
More precisely, for the pair $(\rho, v)$, we derive a variational formulation for \eqref{eq:sim:assos_density_bulk} incorporating divergence-free test functions, an energy equality and a weak formulation of the pure transport of the interface \eqref{eq:sim:v_normal_velocity} in terms of a transport equation for $\chi$.

\begin{assumptions}[Existence of smooth solutions]\label{ass:sim:smoothness}
  Let the following conditions be satisfied.
  \begin{enumerate}
  \item \textbf{Regularity of initial interface.}
    $\Gamma^{(i)}$ is a $C^2$-hypersurface, inducing a disjoint partition $\Omega = \Omega^{-,(i)} \cup \Gamma^{(i)} \cup \Omega^{+,(i)}$, such that
    \begin{equation*}
      \Gamma^{(i)} = \partial \Omega^{-,(i)} \subset\subset \Omega
      \ \text{and} \
      \Omega^{+,(i)}
      = \Omega \setminus \overline{\Omega^{-,(i)}}
      = \Omega \setminus( \Omega^{-,(i)} \cup \Gamma^{(i)} ).
    \end{equation*}
    Define the initial associated density function $\rho^{(i)}\colon \Omega \rightarrow \R$ by
    \begin{equation*}
      \rho^{(i)}(x) = \beta_1 \chi_{\Omega^{-,(i)}}(x) + \beta_2 \chi_{\Omega^{+,(i)}}(x) \ \text{for} \ x\in \Omega
    \end{equation*}
    and define $\chi^{(i)}\colon \Omega \rightarrow \R$ by
    \begin{equation}\label{eq:def_chi_ic}
      \chi^{(i)}(x) = \chi_{\Omega^{-,(i)}}(x) \ \text{for} \ x\in \Omega.
    \end{equation}
    
  \item \textbf{Regularity of initial velocity.}
    $v^{(i)} \colon \Omega \rightarrow \R^n$ belongs to $C^0(\overline \Omega)^n$.
    Additionally, the restrictions $\left. v^{(i)} \right|_{\Omega^{\pm, (i)}}$ to $\Omega^{\pm, (i)}$ satisfy
    \begin{equation*}
      \left. v^{(i)} \right|_{\Omega^{\pm, (i)}} \in C^1( \overline{\Omega^{\pm}_0} )^n
      \ \text{and} \ 
      \left. \dive( v^{(i)} )\right|_{\Omega^{\pm, (i)}} = 0.
    \end{equation*}

  \item \textbf{Existence of smooth solutions.}
    $(v,p,\Gamma)$ is a solution triplet satisfying equations \eqref{eq:sim:bulk1}--\eqref{eq:sim:ic} with the following regularity properties.
    \begin{enumerate}
    \item \textbf{Regularity of velocity and pressure.}
      There exist open sets $U^-,U^+ \subset \R^{n+1}$ with
      \begin{equation*}
        \bigcup_{t\in [0,T]} { \Big(\overline{\Omega^\pm(t)} \times \{t\} \Big) } \subset U^\pm
      \end{equation*}
      as well as functions $v^\pm \in C^2(U^\pm)^n$ and $p^\pm \in C^1(U^\pm)$ such that
      \begin{equation*}
        v=v^\pm \ \text{and} \ p = p^\pm \ \text{on} \ \bigcup_{t\in [0,T]} {\Big(\overline{\Omega^\pm(t)} \times \{t\}\Big) }.
      \end{equation*}

    \item \textbf{Regularity of interface.}
      $( \Gamma(t) )_{t \in [0,T] }$ is a family of evolving hypersurfaces in the sense of Definition \ref{def:evolving_hypersurfaces} such that
      $\Omega^-(t) \cup \Gamma(t) \cup \Omega^+(t)$ is a pairwise disjoint partition of $\Omega$ and 
      $\Gamma(t) = \partial \Omega^-(t) \subset\subset \Omega$ for all $t \in [0,T]$.
      Additionally, for $\Gamma = \bigcup_{t\in [0,T]} \big( \Gamma(t) \times \{ t \} \big)$, let $\nu^- \in C^0(\Gamma)^n$ be such that $\nu^-(\cdot,t)$ is the unit-normal field pointing outward to $\Omega^-(t)$ for all $t\in [0,T]$.
    
    \end{enumerate}
  \end{enumerate}
\end{assumptions}
%
\subsection{Variational Formulation}\label{ssec:sim:var_formulation}
%
In the spirit of the theory of the incompressible Navier--Stokes equations; see for example~\cite{boyer_fabrie, sohr}, we will use divergence-free
test functions in the weak formulation.
This choice leads to a weak formulation lacking the pressure function.
In order to justify this approach, one has to reconstruct the pressure from the weak formulation.

  For the treatment of time derivatives in \eqref{eq:sim:bulk1} and \eqref{eq:sim:bulk2} and for later use, we provide the following consequences of the transport theorem (Theorem~\ref{thm:transport_theorem}).
  
  \begin{lemma}[Transport identities]\label{lem:transport_identities}
    Suppose that Assumptions~\ref{ass:sim:smoothness} are valid.
    Then, for every $t\in(0,T)$ and every $\psi \in C^\infty_{(0)}(\Omega \times [0,T))^n$, the following statements hold true.
    \begin{enumerate}
    \item
      $\tfrac{\mathrm d}{\mathrm d t} \int_{\Omega} \rho v \cdot \psi \,\mathrm{d} x
      = \int_{\Omega \setminus \Gamma(t)}\rho \partial_t ( v \cdot \psi ) \,\mathrm{d} x + (\beta_1 - \beta_2) \int_{\Gamma(t)}  V (v \cdot \psi)  \,\mathrm{d} \Hm^{n-1}(x)$.
     \item
      $\tfrac{\mathrm d}{\mathrm d t} \int_{\Omega} \rho \left| v \right|^2  \,\mathrm{d} x
      = \int_{\Omega \setminus \Gamma(t)} \rho \partial_t \left| v \right|^2 \,\mathrm{d} x + (\beta_1 - \beta_2) \int_{\Gamma(t)}  V \left| v \right|^2 \,\mathrm{d} \Hm^{n-1}(x)$.
    \end{enumerate}
  \end{lemma}
  
  \begin{proof}
In view of \eqref{eq:sim:velocity_jump},
we simply write $v(t) = v^+(t) = v^-(t)$ on $\Gamma(t)$, where
    \begin{equation*}
      v^+(x,t) = \lim_{\xi \searrow 0} v(x+\xi \nu^-(x,t),t)
      \ \text{and} \
      v^-(x,t) = \lim_{\xi \searrow 0} v(x-\xi \nu^-(x,t),t).
    \end{equation*}
    Let $\psi \in C^\infty_{(0)}(\Omega \times [0,T))^n$.
    To prove the first statement, we apply Theorem~\ref{thm:transport_theorem} to obtain
    \begin{equation*}
      \frac{\mathrm d}{\mathrm d t} \int_{\Omega^-(t)} v \cdot \psi \,\mathrm{d} x
      = \int_{\Omega^-(t)} \partial_t ( v \cdot \psi )  \,\mathrm{d} x
        + \int_{\Gamma(t)}  V (v \cdot \psi)  \,\mathrm{d} \Hm^{n-1}(x)
    \end{equation*}
    and, likewise,
    \begin{equation*}
      \frac{\mathrm d}{\mathrm d t} \int_{\Omega^+(t)} v \cdot \psi \,\mathrm{d} x
      = \int_{\Omega^+(t)} \partial_t ( v \cdot \psi )  \,\mathrm{d} x
        - \int_{\Gamma(t)}  V (v \cdot \psi)  \,\mathrm{d} \Hm^{n-1}(x).
    \end{equation*} 

    Recalling the definition of $\rho$ from \eqref{eq:def_rho}, we infer that
    \begin{equation*}
    \begin{split}
       \frac{\mathrm d}{\mathrm d t} \int_{\Omega} \rho v \cdot \psi \,\mathrm{d} x
      ={}& \int_{\Omega \setminus \Gamma(t)}\rho \partial_t ( v \cdot \psi ) \,\mathrm{d} x + (\beta_1 - \beta_2) \int_{\Gamma(t)}  V (v \cdot \psi)  \,\mathrm{d} \Hm^{n-1}(x).
    \end{split}
    \end{equation*}
    The second claim now follows analogously, with $v$ taking the role of $\psi$.
  \end{proof}
    
   
  \begin{proposition}[Weak differentiability of $v$]\label{prop:u_weakly_diff}
    Let $t\in (0,T)$.
    If Assumptions~\ref{ass:sim:smoothness} are satisfied, then $v(t)$ is weakly differentiable in $\Omega$.
  \end{proposition}
  
  \begin{proof}
    Let $t\in (0,T)$.
    In view of Assumptions~\ref{ass:sim:smoothness}, there holds
    \begin{equation*}
      \left. v(t) \right|_{\Omega^-(t)} \in C^2( \overline{ \Omega^-(t) } )^n 
      \ \text{and} \
      \left. v(t) \right|_{\Omega^+(t)} \in C^2( \overline{ \Omega^+(t) } )^n.
    \end{equation*}
    For any $i=1,\dots,n$ and any $\psi \in C^\infty_0( \Omega )^n$, partial integration yields
    \begin{equation*}
    \begin{split}
      \int_\Omega v_i(t) \dive(\psi) \,\mathrm{d} x
      ={}& \int_{\Omega^-(t)} v_i(t) \dive(\psi) \,\mathrm{d} x + \int_{\Omega^+(t)} v_i(t) \dive(\psi) \,\mathrm{d} x\\
      ={}& - \int_{\Omega \setminus \Gamma(t)} \nabla v_i(t) \cdot \psi \,\mathrm{d} x - \int_{\Gamma(t)} [ v_i(t) ] \psi \cdot \nu^- \,\mathrm{d} \Hm^{n-1}(x).      
    \end{split}
    \end{equation*}
    Since, by \eqref{eq:sim:velocity_jump}, there holds $[ v_i(t) ] = 0$, the claim follows.
  \end{proof}

  The following weak concept of mean curvature will be useful for obtaining a variational formulation of \eqref{eq:sim:pressure_jump}.

  \begin{lemma}[Weak-mean-curvature functional]\label{lem:weak_curvature}
    Let $t\in(0,T)$ and suppose that Assumptions~\ref{ass:sim:smoothness} are satisfied.
    For every $\psi \in C^1(\Omega)^n$ with $\dive( \psi ) = 0$ in $\Omega$, there holds
    \begin{equation}\label{eq:weak_curvature}
      \int_{\Gamma(t)} \kappa(t) \nu^-(t) \cdot \psi \,\mathrm{d}\Hm^{n-1}(x)
      =  \int_{\Gamma(t)} \nu^-(t) \otimes \nu^-(t) : \nabla\psi \,\mathrm{d}\Hm^{n-1}(x).
    \end{equation}
  \end{lemma}

  \begin{proof}
    Let $t\in(0,T)$ and $\psi \in C^1(\Omega)^n$ with $\dive( \psi ) = 0$ be arbitrary. 
    We apply the integration-by-parts formula \eqref{eq:int_parts_surface} to $f = \psi_i$ and sum over $i=1,\dots,n$.
    Denoting $\kappa = \kappa(t)$, $\nu^- = \nu^-(t)$ and $\Gamma = \Gamma(t)$, as $\psi$ is divergence free, this implies
    \begin{equation*}
      \int_{\Gamma} \kappa \nu^- \cdot \psi \,\mathrm{d}\Hm^{n-1}(x)
      =  - \int_{\Gamma} \dive_{\Gamma} (\psi) \,\mathrm{d}\Hm^{n-1}(x)
      =  \int_{\Gamma} \nu^- \otimes \nu^- : \nabla \psi \,\mathrm{d}\Hm^{n-1}(x).
    \end{equation*}
  \end{proof}
    
Note that the right-hand side of \eqref{eq:weak_curvature} is well-defined if $\Gamma$ is merely the reduced or the essential boundary of a set of finite perimeter.
Then one has to interpret $\nu^-$ as generalised inner (or outer) normal to $\Gamma$.

\begin{lemma}[Weak form of linear-momentum balance]\label{lem:sim:var_form}
  Let Assumptions~\ref{ass:sim:smoothness} hold true.
  For every $\psi \in C_0^\infty( [0,T) ; C_{0,\sigma}^\infty(\Omega) )$, there holds
  \begin{equation}\label{eq:sim:derivation_var_form}
  \begin{split}
    &\int_0^T \!\int_{\Omega} \rho v \cdot \partial_t\psi + \rho v \otimes v:\nabla \psi - 2 \mu( \rho ) D v : D \psi \,\mathrm{d}x \,\mathrm{d}t\\
    ={}& - \int_{\Omega} \rho^{(i)} v^{(i)} \cdot \psi(0)\,\mathrm{d}x - 2 \stc \int_0^T \!\int_{\Gamma(t)} \nu^- \otimes \nu^- : \nabla \psi \,\mathrm{d}\Hm^{n-1}(x) \,\mathrm{d}t.
  \end{split}
  \end{equation}
\end{lemma}

\begin{proof}
  Multiplying \eqref{eq:sim:assos_density_bulk} by $\psi \in C_0^\infty( [0,T) ; C_{0,\sigma}^\infty( \Omega ) )$ and integrating with respect to space and time leads to
  \begin{equation}\label{eq:bulk_mult_psi}
    \int_0^T \!\int_{\Omega \setminus \Gamma(t)} \big( \rho \partial_t v + \rho (v \cdot \nabla) v - 2 \mu( \rho ) \dive( D v ) + \nabla p  \big) \cdot \psi  \,\mathrm{d}x \,\mathrm{d}t = 0.
  \end{equation}
  Applying the first statement of Lemma~\ref{lem:transport_identities} to deal with the time derivative leads to
  \begin{equation}
  \begin{split}
    &\int_0^T \!\int_{\Omega \setminus \Gamma(t)} \rho \partial_t v \cdot \psi \,\mathrm{d}x \,\mathrm{d}t + (\beta_1 - \beta_2) \int_0^T \!\int_{\Gamma(t)}  V ( v \cdot \psi ) \,\mathrm{d}\Hm^{n-1}(x) \,\mathrm{d}t\\
    ={}& - \int_0^T \!\int_{\Omega} \rho v \cdot \partial_t \psi \,\mathrm{d}x \,\mathrm{d}t + \int_0^T \,\left( \frac{\mathrm d}{\mathrm d t} \int_{\Omega} \rho v \cdot \psi \,\mathrm{d} x \right) \,\mathrm{d}t\\
    ={}& - \int_0^T \!\int_{\Omega} \rho v \cdot \partial_t \psi \,\mathrm{d}x \,\mathrm{d}t  - \int_{\Omega} \rho^{(i)} v^{(i)} \cdot \psi(0)\,\mathrm{d}x.
  \end{split}
  \end{equation}
  To each of the remaining terms in~\eqref{eq:bulk_mult_psi}, we shall apply the integration-by-parts formula on the spatial domains $\Omega^-(t)$ and $\Omega^+(t)$:
  By \eqref{eq:sim:divergence_free} and \eqref{eq:sim:velocity_jump}, we infer
  \begin{equation}
  \begin{split}
     &\int_{\Omega \setminus \Gamma(t)} \rho ( (v \cdot \nabla) v ) \cdot \psi \,\mathrm{d}x\\
   ={}&\beta_1 \int_{\Omega^-(t)} \dive( v \otimes v ) \cdot \psi \,\mathrm{d}x + \beta_2 \int_{\Omega^+(t)} \dive( v \otimes v ) \cdot \psi \,\mathrm{d}x\\
    ={}&- \int_{\Omega}  \rho v \otimes v : \nabla \psi \,\mathrm{d}x + ( \beta_1 -\beta_2 ) \int_{\Gamma(t)} (v \cdot \nu^-) v \cdot \psi \,\mathrm{d}\Hm^{n-1}(x).
  \end{split}
  \end{equation}
  Using Proposition~\ref{prop:u_weakly_diff} and $Dv : \nabla \psi = Dv : D\psi$, we analogously obtain that
  \begin{equation}\label{eq:bulk_mult_psi_term3}
  \begin{split}
    &\int_{\Omega \setminus \Gamma(t)} \mu( \rho ) \dive ( D v ) \cdot \psi \,\mathrm{d}x\\
    ={}&-  \int_{\Omega}\mu( \rho ) D v : D \psi \,\mathrm{d}x - \int_{\Gamma(t)} [ \mu( \rho )  Dv ]  \nu^- \cdot \psi \,\mathrm{d}\Hm^{n-1}(x).
  \end{split}
  \end{equation}
  In view of $\dive( \psi ) = 0$, we have 
  \begin{equation}\label{eq:bulk_mult_psi_term4}
  \begin{split}
    \int_{\Omega \setminus \Gamma(t)} \nabla p \cdot \psi \,\mathrm{d}x
    ={}& - \int_{\Gamma(t)} [ p ] \nu^- \cdot \psi \,\mathrm{d}\Hm^{n-1}(x).
  \end{split}
  \end{equation}
  Now combining \eqref{eq:bulk_mult_psi}--\eqref{eq:bulk_mult_psi_term4} leads to 
  \begin{equation*}
  \begin{split}
    &\int_0^T \!\int_{\Omega} \rho v \cdot \partial_t\psi + \rho v \otimes v:\nabla \psi - 2 \mu( \rho ) D v : D \psi \,\mathrm{d}x \,\mathrm{d}t
+ \int_{\Omega} \rho^{(i)} v^{(i)} \cdot \psi(0)\,\mathrm{d}x\\
    ={}& - (\beta_1 - \beta_2) \int_0^T \!\int_{\Gamma(t)}  (V - v\cdot \nu^-) v \cdot \psi \,\mathrm{d}\Hm^{n-1}(x) \,\mathrm{d}t\\
     & + \int_0^T \!\int_{\Gamma(t)} \left[2 \mu ( \rho ) D v \nu^- - p \nu^-\right]  \cdot \psi \,\mathrm{d}\Hm^{n-1}(x) \,\mathrm{d}t\\
    ={}&- 2 \stc \int_0^T \!\int_{\Gamma(t)} \kappa \nu^- \cdot \psi \,\mathrm{d}\Hm^{n-1}(x) \,\mathrm{d}t,
  \end{split}
  \end{equation*}
  where the last identity follows by \eqref{eq:sim:v_normal_velocity} and \eqref{eq:sim:pressure_jump}.
  Finally, Lemma~\ref{lem:weak_curvature} yields \eqref{eq:sim:derivation_var_form}.
  \end{proof}

\subsection{Energy Equality}\label{ssec:sim:energy_equality}

In an analogous manner to Lemma~\ref{lem:sim:var_form}, we may derive the following energy identity. 

\begin{lemma}[Energy equality and a priori bounds]\label{lem:sim:energy_eq}
  Let Assumptions~\ref{ass:sim:smoothness} hold true.
  For all $\tau_1 ,\tau_2 \in [ 0,T ]$ such that $\tau_1 \leq \tau_2$, the following energy equality is satisfied.
  \begin{equation}\label{eq:sim:energy_equality}
  \begin{split}
     &2 \stc \Hm^{n-1}{(\Gamma(\tau_2))}
     +\tfrac{1}{2} \int_{\Omega} \rho(\tau_2) \left| v(\tau_2) \right|^2 \,\mathrm{d}x
     +2 \int_{\tau_1}^{\tau_2} \!\int_{\Omega} \mu( \rho ) \left|D v \right|^2 \,\mathrm{d}x \,\mathrm{d}t\\
    ={}&2 \stc \Hm^{n-1}{(\Gamma(\tau_1))} + \tfrac{1}{2} \int_{\Omega} \rho(\tau_1) \left| v(\tau_1) \right|^2 \,\mathrm{d}x.
  \end{split}
  \end{equation}
  Moreover, if the initial energy
  \begin{equation}\label{eq:def_initial_energy}
    E^{(i)}
    =2 \stc \Hm^{n-1}(\Gamma^{(i)}) + \tfrac{1}{2} \int_{\Omega}\rho^{(i)} \left|v^{(i)} \right|^2 \,\mathrm{d}x
  \end{equation}
  is finite, then there holds
  \begin{equation*}
    v \in L^\infty(0,T; L^2_{\sigma}(\Omega)) \cap L^2(0,T; H^1_0(\Omega)^n)
    \ \text{and} \
    \rho \in L^\infty(0,T; BV(\Omega, \{ \beta_1, \beta_2\} ) ).
  \end{equation*}
\end{lemma}

\begin{proof}
  Let $\tau_1 ,\tau_2 \in [ 0,T ]$ be such that $\tau_1 \leq \tau_2$.
  We multiply \eqref{eq:sim:assos_density_bulk} by $v$ and integrate with respect to space and time.
  This leads to
  \begin{equation}\label{eq:bulk_mult_v}
    \int_{\tau_1}^{\tau_2} \!\int_{\Omega \setminus \Gamma(t)} \big( \rho \partial_t v + \rho (v \cdot \nabla) v - 2 \mu( \rho ) \dive ( D v ) + \nabla p  \big) \cdot v  \,\mathrm{d}x \,\mathrm{d}t = 0.
  \end{equation}
  We shall evaluate the integral expression successively.
  For the treatment of the time derivative, we apply Lemma~\ref{lem:transport_identities} to obtain
  \begin{equation}\label{eq:bulk_mult_v_term1}
  \begin{split}
    2 \int_{\tau_1}^{\tau_2} \!\int_{\Omega \setminus \Gamma(t)} \rho \partial_t v \cdot v \,\mathrm{d}x \,\mathrm{d}t
    ={}& \int_{\tau_1}^{\tau_2} \!\int_{\Omega \setminus \Gamma(t)} \rho \partial_t \left| v \right|^2 \,\mathrm{d}x \,\mathrm{d}t\\
    ={}& \int_{\Omega} \rho(\tau_2) \left| v(\tau_2) \right|^2 \,\mathrm{d}x -  \int_{\Omega} \rho(\tau_1) \left| v(\tau_1) \right|^2 \,\mathrm{d}x\\
    &-  (\beta_1 - \beta_2) \int_{\tau_1}^{\tau_2} \!\int_{\Gamma(t)}  V  \left| v \right|^2  \,\mathrm{d}\Hm^{n-1}(x) \,\mathrm{d}t.
  \end{split}
  \end{equation} 
  For the treatment of the remaining terms in~\eqref{eq:bulk_mult_v}, we shall repeatedly integrate by parts with respect to the spatial variable for fixed $t\in ( \tau_1,\tau_2 )$:
  for the computation of the second term in~\eqref{eq:bulk_mult_v}, we use that,
in view of~\eqref{eq:sim:divergence_free}, there holds $(( v \cdot \nabla) v) \cdot v = \dive( v \otimes v ) \cdot v = \tfrac{1}{2} \dive( \left| v\right|^2 v )$ in $\Omega \setminus \Gamma(t)$, which implies
  \begin{equation}\label{eq:bulk_mult_v_term2}
  \begin{split}
    2 \int_{\Omega \setminus \Gamma(t)} \rho ( (v \cdot \nabla) v ) \cdot v \,\mathrm{d}x
    ={}& ( \beta_1 - \beta_2 ) \int_{\Gamma(t)} \left| v\right|^2 v \cdot \nu^- \,\mathrm{d}\Hm^{n-1}(x).
  \end{split}
  \end{equation}
  Proceeding as in \eqref{eq:bulk_mult_psi_term3} and using $Dv : \nabla v = \left| Dv \right|^2$ leads to
  \begin{equation}\label{eq:bulk_mult_v_term3}
  \begin{split}
    &\int_{\Omega \setminus \Gamma(t)} \mu( \rho ) \dive ( D v ) \cdot v \,\mathrm{d}x\\
    ={}&-  \int_{\Omega} \mu( \rho )  \left| D v \right|^2 \,\mathrm{d}x - \int_{\Gamma(t)} [ \mu( \rho )  Dv ]  \nu^- \cdot v \,\mathrm{d}\Hm^{n-1}(x).
  \end{split}
  \end{equation}
  To treat the pressure term in \eqref{eq:bulk_mult_v}, we may again use
\eqref{eq:sim:divergence_free}.
  Using calculations as in \eqref{eq:bulk_mult_psi_term4}, we infer that
  \begin{equation}\label{eq:bulk_mult_v_term4}
    \int_{\Omega \setminus \Gamma(t)} \nabla p \cdot v \,\mathrm{d}x
    = - \int_{\Gamma(t)} [ p ]  \nu^- \cdot v \,\mathrm{d}\Hm^{n-1}(x).
  \end{equation}
  Now we may combine \eqref{eq:bulk_mult_v}--\eqref{eq:bulk_mult_v_term4}.
  Altogether, by \eqref{eq:sim:v_normal_velocity} and \eqref{eq:sim:pressure_jump}, we obtain
  \begin{equation*}
  \begin{split}
    &\frac{1}{2}\int_{\Omega} \rho(\tau_2) \left| v(\tau_2) \right|^2 \,\mathrm{d}x -\frac{1}{2} \int_{\Omega} \rho(\tau_1) \left| v(\tau_1) \right|^2 \,\mathrm{d}x
    + 2\int_{\tau_1}^{\tau_2} \!\int_{\Omega} \mu( \rho ) \left|D v \right|^2 \,\mathrm{d}x \,\mathrm{d}t\\
    ={}&\frac{1}{2}( \beta_1 - \beta_2 ) \int_{\tau_1}^{\tau_2} \!\int_{\Gamma(t)} \left| v\right|^2 (V - v \cdot \nu^-) \,\mathrm{d}\Hm^{n-1}(x) \,\mathrm{d}t\\
     & - \int_{\tau_1}^{\tau_2} \!\int_{\Gamma(t)} [ 2 \mu( \rho ) Dv \nu^- - p  \nu^- ] \cdot v \,\mathrm{d}\Hm^{n-1}(x) \,\mathrm{d}t\\
    ={}& 2 \stc \int_{\tau_1}^{\tau_2} \!\int_{\Gamma(t)} \kappa v\cdot \nu^- \,\mathrm{d}\Hm^{n-1}(x) \,\mathrm{d}t\\
    ={}& 2 \stc \int_{\tau_1}^{\tau_2} \!\int_{\Gamma(t)} \kappa V \,\mathrm{d}\Hm^{n-1}(x) \,\mathrm{d}t.
   \end{split}
   \end{equation*}
Now \eqref{eq:sim:energy_equality} follows by observing that, in view of Theorem~\ref{thm:transport_theorem}, there holds
 \begin{equation*}
   \begin{split}
     \int_{\tau_1}^{\tau_2} \!\int_{\Gamma(t)} \kappa V \,\mathrm{d}\Hm^{n-1}(x) \,\mathrm{d}t
     ={}& - \int_{\tau_1}^{\tau_2} \tfrac{\mathrm d}{\mathrm d t} \Hm^{n-1}( \Gamma( t ) ) \,\mathrm{d}t\\
     ={}& \Hm^{n-1}( \Gamma(\tau_1) ) - \Hm^{n-1}( \Gamma(\tau_2) ).
   \end{split}
 \end{equation*}
  \par
  Suppose that the initial energy $E^{(i)}$, defined by \eqref{eq:def_initial_energy}, is finite and let $t\in[0,T]$.
  Recall from \eqref{eq:def_rho} that there holds $\rho(t) \in \{ \beta_1, \beta_2\}$ a.e.\ in $\Omega$.
  Making in \eqref{eq:sim:energy_equality} the choice $\tau_1 = 0$ and $\tau_2 = t$ then implies
  \begin{equation*}
    \tfrac{1}{2} \beta_1 \int_{\Omega} \left| v(t) \right|^2 \,\mathrm{d}x
    \leq E^{(i)}.
  \end{equation*}
  Hence $v \in L^\infty(0,T ; L^2(\Omega)^n )$.
  Similarly, we obtain that
  \begin{equation*}
    \min\{ \mu(\beta_1) , \mu(\beta_2) \} \int_0^T \!\int_{\Omega} \left|D v \right|^2 \,\mathrm{d}x \,\mathrm{d}t
    \leq \int_0^T \!\int_{\Omega} \mu( \rho ) \left| D v \right|^2 \,\mathrm{d}x \,\mathrm{d}t
    \leq \tfrac{1}{2} E^{(i)}.
  \end{equation*}
  Due to the boundary condition \eqref{eq:sim:bc}, and using Korn's inequality \cite[Theorem~1.33]{roubicek}, we infer that $v \in L^2(0,T; H^1_0(\Omega)^n)$.
  \par
  In the remainder of the proof we fix $t\in(0,T)$.
  In view of \eqref{eq:sim:divergence_free} and Lemma~\ref{lem:sigma_characterisations}, $v$ belongs to $L^\infty( 0,T ; L^2_{\sigma}(\Omega) )$.
  To explore the regularity of $\rho$, we recall that, in view of \eqref{eq:def_rho}, for every $t\in (0,T)$, there holds $\rho(t) \in \{ \beta_1, \beta_2\}$ a.e.\ in $\Omega$ and, in particular, $\rho$ belongs to $L^\infty( 0,T ; L^\infty(\Omega) )$.
  Additionally, for any $\psi \in C^\infty_0(\Omega)^n$, we have
  \begin{equation*}
  \begin{split}
    \langle  \nabla \rho(t) , \psi \rangle_{\mathcal{D}(\Omega)^n}
    &= - \int_{\Omega}  \rho(t) \dive ( \psi ) \,\mathrm{d} x\\
    &= - ( \beta_1 - \beta_2 ) \int_{\Omega^-(t)} \dive ( \psi ) \,\mathrm{d} x\\
    &= - ( \beta_1 - \beta_2 ) \int_{\Gamma(t)} \psi \cdot \nu^-(t) \,\mathrm{d} \Hm^{n-1}(x).
  \end{split}
  \end{equation*}

  Consequently, $\nabla \rho(t)$ is a finite Radon measure and there holds
  \begin{equation*}
  \begin{split}
    \| \nabla \rho(t) \|_{\mathcal{M}(\Omega)}
    &= \sup \left\{ \int_{\Omega} \rho(t) \dive (\psi) \,\mathrm{d} x \,:\, \psi\in C_0^1(\Omega)^n, \ \| \psi \|_{\infty} \leq 1 \right\} \\
    &= ( \beta_2 - \beta_1 )\sup \left\{ \int_{\Omega^-(t)} \dive (\psi) \,\mathrm{d} x \,:\, \psi\in C_0^1(\Omega)^n, \ \| \psi \|_{\infty} \leq 1 \right\} \\
    &= ( \beta_2 - \beta_1 ) \Hm^{n-1}( \partial \Omega^-(t) \cap \Omega ).
  \end{split}
  \end{equation*}
  Due to Assumptions~\ref{ass:sim:smoothness}, $\Omega^-(t)$ has a Lipschitz boundary and $\Omega^-(t) \subset\subset \Omega$.
  Then,
we get
  \begin{equation}\label{eq:variation_rho}
    \| \nabla \rho(t) \|_{\mathcal{M}(\Omega)}
    = ( \beta_2 - \beta_1 )\Hm^{n-1}( \Gamma(t) ).
  \end{equation}
  Finally, from the energy equality \eqref{eq:sim:energy_equality}, it follows that $\| \nabla \rho(t) \|_{\mathcal{M}(\Omega)}$ is uniformly bounded in $t$.
  Altogether, we have proven that $\rho \in L^\infty(0,T; BV(\Omega, \{ \beta_1, \beta_2\} ) )$.
\end{proof}

\subsection{Transport Equation} \label{ssec:sim:transport_eq}

The interface condition \eqref{eq:sim:v_normal_velocity} can be expressed by the following transport equation for $\chi$ in distributional form, cf.\ \cite[Section~2.5]{abels}.

\begin{lemma}[Transport equation]\label{lem:sim:transport_eq}
  Let Assumptions~\ref{ass:sim:smoothness} hold true.
  Then, for all $\varphi \in C_{(0)}^\infty(\overline{\Omega} \times [0,T))$, there holds
  \begin{equation}\label{eq:transport_var}
    \int_0^T \!\int_{\Omega} \chi ( \partial_t \varphi + v \cdot \nabla \varphi) \,\mathrm{d} x \,\mathrm{d} t + \int_{\Omega} \chi^{(i)}(x)\varphi(0) \,\mathrm{d} x = 0.
  \end{equation}
\end{lemma}

\begin{proof}
  Let $\varphi \in C_{(0)}^\infty(\overline{\Omega} \times [0,T))$.
  Applying Theorem~\ref{thm:transport_theorem} to $\varphi$ and integrating with respect to time yields
  \begin{equation}\label{eq:interface_velocity_weak}
    \int_0^T \!\int_{\Gamma(t)} V \varphi \,\mathrm{d} \Hm^{n-1}(x) \,\mathrm{d} t
    = - \int_0^T \!\int_{\Omega^-(t)} \partial_t \varphi \,\mathrm{d} x \,\mathrm{d} t - \int_{\Omega^-(0)} \varphi(0) \,\mathrm{d} x.
  \end{equation}
  
  Since $\dive(v(t)) = 0$ in $\Omega^-(t)$ by \eqref{eq:sim:divergence_free}, we conclude that
  \begin{equation*}
  \begin{split}
    \int_0^T \!\int_{\Gamma(t)} v\cdot \nu^- \varphi \,\mathrm{d} \Hm^{n-1}(x) \,\mathrm{d} t
    &= \int_0^T \!\int_{\Omega^-(t)} \dive(v \varphi) \,\mathrm{d} x \,\mathrm{d} t\\
    &= \int_0^T \!\int_{\Omega^-(t)} v \cdot \nabla \varphi \,\mathrm{d} x \,\mathrm{d} t.
  \end{split}
  \end{equation*}

  Recalling \eqref{eq:sim:v_normal_velocity}, we use that $V = v \cdot \nu^-$ on $\Gamma(t)$ to obtain
  \begin{equation*}
    \int_0^T \!\int_{\Omega^-(t)} \partial_t \varphi \,\mathrm{d} x \,\mathrm{d} t + \int_{\Omega^-(0)} \varphi(0) \,\mathrm{d} x + \int_0^T \!\int_{\Omega^-(t)} v \cdot \nabla \varphi \,\mathrm{d} x \,\mathrm{d} t
    = 0.
  \end{equation*}
  As $\chi(t)$ and $\chi^{(i)}$ are the characteristic functions of $\Omega^-(t)$ and $\Omega^{-,(i)}=\Omega^-(0)$, respectively, see~\eqref{eq:def_chi} and \eqref{eq:def_chi_ic}, the identity \eqref{eq:transport_var} follows.
  This finishes the proof.
\end{proof}

The previous result motivates the following definition.

\begin{definition}[Weak solutions of the transport equation]\label{def:weak_sol_transport_eq}
  For prescribed functions $v\in L^2(0,T;L^2_{\sigma} (\Omega))$ and $\chi^{(i)} \in L^\infty(\Omega)$, $\chi \in L^\infty(\Omega \times (0,T))$ is called a weak solution of the transport equation
  \begin{equation}\label{eq:sim:transport_equation}
  \begin{aligned}
    \partial_t \chi + v \cdot \nabla \chi &= 0 &&\text{in} \ \Omega \times (0,T),\\
    \chi(0) &= \chi^{(i)} &&\text{in} \ \Omega
  \end{aligned}
  \end{equation}
  provided that for every $\varphi \in C_{(0)}^\infty(\overline{\Omega} \times [0,T))$, \eqref{eq:transport_var} holds true.
\end{definition}

\subsection{The Weak Formulation}

We seek to introduce a weak formulation for \eqref{eq:sim:bulk1}--\eqref{eq:sim:ic}.
To this end, we restrict the class of weak solutions to pairs $(\rho,v)$ satisfying the energy inequality~\eqref{eq:sim:energy_equality}.
For well-prepared initial data $(\rho^{(i)},v^{(i)})$, this suggests the regularity classes $\rho \in L^\infty(0,T; BV(\Omega, \{ \beta_1, \beta_2\} ) )$ and $v \in L^\infty(0,T; L^2_{\sigma}(\Omega)) \cap L^2(0,T; H^1_0(\Omega)^n)$.
For a.e.\ $t\in(0,T)$, there exist a measurable set $\Omega^-(t) \subset \Omega$ and an induced characteristic function $\chi(t) \in BV(\Omega,\{0,1\})$  of $\rho(t)$ such that, a.e.\ in $\Omega$, there holds
\begin{equation*}
  \chi(t)
  = \chi_{\Omega^-(t)}
  = \frac{\rho(t) - \beta_2}{\beta_1 - \beta_2}.
\end{equation*}
Here and subsequently, we refer $\Omega^-(t)$ to as measure-theoretic representative set of $\rho(t)$.
This, in turn, leads to the representation
\begin{equation*}
  \rho(t)
  = ( \beta_1 - \beta_2 ) \chi_{\Omega^-(t)} + \beta_2
  = ( \beta_1 - \beta_2 ) \chi(t) + \beta_2.
\end{equation*}
Notice that this procedure makes the identity $\Omega^-(t) = \left\{ x \in \Omega \,:\, \rho(t) = \beta_1 \right\}$ well-defined in a measure-theoretic sense.
As $\Omega^-(t)$ is of bounded variation, we may define the interface $\Gamma(t)$ by $\Gamma(t) = \partial^* (\Omega^-(t)) \cap \Omega$,
where $\partial^* (\Omega^-(t))$ denotes the reduced boundary of $\Omega^-(t)$.
Hence the variational formulation~\eqref{eq:sim:derivation_var_form} remains meaningful if we understand the outer unit normal $\nu^-$ in the (measure-theoretic) sense of the generalised outer unit normal given by
\begin{equation*}
  \nu^-(x,t) = - \lim_{\delta \rightarrow 0} \frac{\nabla \chi_{\Omega^-(t)} (B_\delta(x))}{\left| \nabla \chi_{\Omega^-(t)} \right|(B_\delta(x))}
  \ \text{for} \ x \in \Gamma(t).
\end{equation*}
Additionally, we require $\chi$ to solve the corresponding transport equation in the sense of Definition~\ref{def:weak_sol_transport_eq}.
and we maintain the assumption that $\Omega^-(t)$ is compactly contained in $\Omega$.
Finally, the results of the
Lemmas~\ref{lem:sim:var_form}--\ref{lem:sim:transport_eq}
motivate the following weak formulation of \eqref{eq:sim:bulk1}--\eqref{eq:sim:ic}.

\begin{definition}[Weak formulation]\label{def:sim:weak_solution}
 Let $\big( \rho^{(i)}, v^{(i)} \big) \in BV(\Omega,\{\beta_1,\beta_2\}) \times H^1_{0,\sigma}(\Omega)$ be prescribed initial data, such that the measure-theoretic representative set $\Omega^-(0)$ of $\rho^{(i)}$ is compactly contained in $\Omega$, i.e., $\Omega^-(0) \subset\subset \Omega$, and $\rho^{(i)}$ has the representation
  \begin{equation*}
    \rho^{(i)}
    = ( \beta_1 - \beta_2 ) \chi_{\Omega^-(0)} + \beta_2
    = ( \beta_1 - \beta_2 ) \chi^{(i)} + \beta_2,
  \end{equation*}
  where $\chi^{(i)}$ is the induced characteristic function of $\rho^{(i)}$ that is given by
  \begin{equation*}
    \chi^{(i)} = \frac{\rho^{(i)} - \beta_2}{\beta_1 - \beta_2} \in BV( \Omega , \{ 0 , 1 \} ).
  \end{equation*}
  Then $(\rho, v)$ is called a weak solution of \eqref{eq:sim:bulk1}--\eqref{eq:sim:ic} with prescribed initial data $( \rho^{(i)} , v^{(i)})$ if the following conditions are fulfilled.
  \begin{enumerate}
   \item \textbf{Regularity of associated density.}
      $\rho \in L^\infty(0,T;BV(\Omega,\{\beta_1,\beta_2\}))$, and the measure-theoretic representative set $\Omega^-(t)$ of $\rho(t)$ is compactly contained in $\Omega$; that is, for a.e.\ $t\in (0,T)$, there holds $\Omega^-(t) \subset\subset \Omega$.

    \item \textbf{Regularity of velocity.}
      $v \in L^\infty(0,T;L^2_{\sigma}(\Omega)) \cap L^2(0,T; H^1_0(\Omega)^n)$.
     
    \item \textbf{Weak form of linear-momentum balance.}
      For each $\psi \in C_0^\infty( [0,T) ; C_{0,\sigma}^\infty(\Omega) )$, there holds
      \begin{equation}\label{eq:sim:variational_formulation}
      \begin{split}
        &\int_0^T \!\int_{\Omega} \rho v \cdot \partial_t\psi + \rho v \otimes v:\nabla \psi  - 2 \mu( \rho ) D v : D \psi \,\mathrm{d}x \,\mathrm{d}t\\
         ={}& - \int_{\Omega} \rho^{(i)} v^{(i)} \cdot \psi(0)\,\mathrm{d}x
        - 2 \stc \int_0^T \!\int_{\Gamma(t)} \nu^- \otimes \nu^- : \nabla \psi \,\mathrm{d}\Hm^{n-1}(x) \,\mathrm{d}t,
      \end{split}
      \end{equation} 
      where $\Gamma(t) = \partial^* (\Omega^-(t))$ is the reduced boundary of $\Omega^-(t)$, and $\nu^-(t)$ denotes the corresponding generalised outer unit normal. 

   \item \textbf{Energy inequality.}
      For a.e.\ $\tau_1 \in [ 0,T )$, including $\tau_1 =0$, there holds
      \begin{equation}\label{eq:sim:energy_inequality}
      \begin{split}
        &2 \stc \Hm^{n-1}( \Gamma(\tau_2) )
        + \tfrac{1}{2} \int_{\Omega} \rho(\tau_2) \left| v(\tau_2) \right|^2 \,\mathrm{d}x    
        + 2 \int_{\tau_1}^{\tau_2} \!\int_{\Omega} \mu( \rho ) \left|D v \right|^2 \,\mathrm{d}x \,\mathrm{d}t \\
      \leq {}&2 \stc \Hm^{n-1}( \Gamma(\tau_1) ) + \tfrac{1}{2} \int_{\Omega} \rho(\tau_1) \left| v(\tau_1) \right|^2 \,\mathrm{d}x
      \end{split}
      \end{equation}
      for all $\tau_2 \in [\tau_1 ,T)$.
      
   \item \textbf{Transport equation.}
   The induced characteristic function $\chi$ given by $\chi = \chi_{\Omega^-(\cdot)}$, that is,
    \begin{equation*}
      \chi =  \frac{\rho - \beta_2}{\beta_1 - \beta_2},
    \end{equation*}
    is a weak solution of the transport equation \eqref{eq:sim:transport_equation} with velocity $v$ and prescribed initial data $\chi^{(i)}$ in the sense of Definition \ref{def:weak_sol_transport_eq}.  
  \end{enumerate}
\end{definition}

From now on, we will always consider weak solutions in the sense of the foregoing definition.
For convenience, for any weak solution $(\rho, v)$, we will use the notation
\begin{equation*}
  \Omega^+(t) = \Omega \setminus ( \Omega^-(t) \cup \Gamma(t) ),
\end{equation*}
where, as in the previous definition, $\Omega^-(t)$ denotes the measure-theoretic representative set of $\rho(t)$ and $\Gamma(t) = \partial^*(  \Omega^-(t) )$.
This means that, via $\Omega = \Omega^-(t) \cup \Gamma(t) \cup \Omega^+(t)$, this notation leads to a pairwise disjoint partition of $\Omega$.
Note that if the set $\Omega^-(t)$ is sufficiently smooth, its topological and reduced boundary coincide, i.e., $\Gamma(t) = \partial^*( \Omega^-(t) ) =  \partial(  \Omega^-(t) )$.
This is consistent with Assumptions~\ref{ass:sim:smoothness}.

\begin{remark}[Energy inequality]\label{rem:weaksol_energyineq}
  The energy inequality \eqref{eq:sim:energy_inequality} restricts the class of weak solutions in Definition~\ref{def:sim:weak_solution}.
  This approach is in the spirit of the theory of weak solutions for the incompressible Navier--Stokes equations: in this case, for $n=2$, weak solutions are unique, whereas, for $n=3$, it can be shown that weak solutions are unique if one weak solution satisfies an additional regularity assumption, referred to as Serrin's condition, cf.\ \cite[Theorem~V.1.5.1]{sohr}.
\end{remark}

\section{Lebesgue and Sobolev Spaces on Time-Dependent Domains}\label{sec:sobolev_time_dep}

We are interested in functions that take values in Lebesgue or Sobolev spaces on time-de\-pend\-ent domains $(\Omega(t))_{t\in[0,T]}$,
cf.\ also \cite{alphonse_evolving_spaces, alphonse_linear_parabolic, naegele, saal}.
%
%
%
%
We require the family $(\Omega(t))_{t\in[0,T]}$ to be parametrised in the following way, cf.\ \cite[Assumption~1.1]{saal}.
  
\begin{assumptions}[Time evolution]\label{ass:reg_omega_t}
  Let $\Omega \subset \R^n$, $n=2,3$, be a bounded domain with boundary $\partial \Omega$ of class $C^3$.
  Assume that the time evolution of the family $(\Omega(t))_{t\in[0,T]} \subset \Omega$ is described via a time-dependent $C^3$-diffeomorphism $\Phi(\cdot;t)\colon \overline{\Omega(0)} \rightarrow \overline{\Omega(t)}$, i.e., for every $t\in [0,T]$, there holds
  \begin{equation*}
    \Omega(t) = \big\{ \Phi(\xi;t) \,:\, \xi \in \Omega(0) \big\}
    \ \text{and} \ 
    \overline{\Omega(t)} = \big\{ \Phi(\xi;t) \,:\, \xi \in \overline{ \Omega(0) } \big\}.
  \end{equation*}
    Denote by $\nu= \nu( \cdot , t)$ the corresponding outer unit normal and by $V=V(\cdot,t)$ the normal velocity of $(\partial \Omega(t))_{t\in [ 0 , T ]}$ with respect to $\nu$.
For $Q \subset \R^n \times [0,T]$, $C^0_{\mathrm b}( Q )$ denotes the set of all bounded, continuous real-valued functions on $Q$ and $C^{3,1}_{\mathrm b}( Q )$ is given by
\begin{equation*}
  \big\{ u \in C^0_{\mathrm b} ( Q ) \,:\, \partial_t^s \partial^\alpha u \in C^0_{\mathrm b}( Q ), \ 1\leq 2s + \left| \alpha \right|_\ast \leq 3, \ s\in \N_0, \ \alpha \in \N^n_0 \big\},
\end{equation*}
where $\left| \alpha \right|_\ast = \alpha_1 + \alpha_2 + \dots + \alpha_n$.

  \begin{enumerate}
  \item \textbf{Regularity of initial domain.}
    The initial domain $\Omega(0) \subset \R^n$ is a bounded domain with $C^3$-boundary $\partial (\Omega(0))$ and let $Q_0 = \Omega(0) \times (0,T)$.
    
  \item \textbf{Regularity of $\Phi$.}
    $\Phi \in C^{3,1}_{\mathrm b}(\overline{Q_0})^n$.

  \item \textbf{Preservation of volume.}
    $\det ( \nabla \Phi(\xi;t) ) = 1$ for all $(\xi,t) \in \overline{Q_0}$.
  \end{enumerate}
\end{assumptions}

\begin{corollary}[Space-time domain]\label{cor:reg_inv_phi}
  Let $\Phi$ be as in Assumptions~\ref{ass:reg_omega_t}.
  Then the function $\varLambda\colon ( \xi ,t ) \mapsto ( \Phi ( \xi  ; t ) , t )$ belongs to $C^{3,1}_{\mathrm b} ( \overline{Q_0} )^{n+1}$.
  Moreover, $\varLambda$ is invertible with inverse function $\varLambda^{-1} \in C^{3,1}_{\mathrm b} ( \overline{\Omega_T} )^{n+1}$, where
  \begin{equation*}
    \Omega_T= \bigcup_{t\in(0,T)} \big(\Omega(t) \times \{ t \}\big) \subset \R^{n+1}.
  \end{equation*}
  In particular, $\Phi^{-1} \in C^{3,1}_{\mathrm b} ( \overline{\Omega_T} )^n$ and $\Omega_T$ has a Lipschitz boundary.
\end{corollary}

\begin{proof}
  As $\Phi \in C^{3,1}_{\mathrm b} ( \overline{Q_0} )^n$ by Assumptions~\ref{ass:reg_omega_t}, it follows that $\varLambda \in C^{3,1}_{\mathrm b} ( \overline{Q_0} )^{n+1}$.
  Moreover, $\varLambda$ is invertible and $\varLambda^{-1}( x ,t ) = ( \Phi^{-1}( x ; t ) , t )$.
  Hence, $\varLambda^{-1} \in C^{3,1}_{\mathrm b} ( \overline{\Omega_T} )^{n+1}$, and thus $\Phi^{-1} \in C^{3,1}_{\mathrm b} ( \overline{\Omega_T } )^n$.
  Observing that $\partial ( Q_0 )$ is Lipschitz and that $\Omega_T = \varLambda ( Q_0 )$ finishes the proof.
\end{proof}

\begin{proposition}[Normal velocity]\label{prop:representation_normal velocity}
  Suppose that Assumptions~\ref{ass:reg_omega_t} hold true.
  Then, for every $(x_0,t_0) \in \bigcup_{t\in[ 0,T ]} \big( \partial \Omega (t) \times \{ t \} \big)$, there holds
  \begin{equation*}
    V(x_0,t_0)
    = ( \partial_t \Phi)( \Phi^{-1} (x_0;t_0) ; t_0 ) \cdot \nu(x_0,t_0).
  \end{equation*}
\end{proposition}

\begin{proof}
  For $t_0 \in [0,T]$, fix $x_0 \in \partial \Omega (t_0)$.
  By Assumptions~\ref{ass:reg_omega_t}, restriction to the respective boundaries yields diffeomorphisms $\Phi^{-1}(\cdot;t_0)\colon \partial \Omega (t_0) \rightarrow \partial \Omega(0)$ and, for $t\in [0,T]$, $\Phi(\cdot;t)\colon \partial \Omega (0) \rightarrow \partial \Omega(t)$.
  Therefore, $t \mapsto \eta(t) = \Phi ( \Phi^{-1}(x_0; t_0) ; t) \in \partial \Omega(t)$ defines a $C^1$-mapping $\eta\colon [0,T] \rightarrow \R^n$ with $\eta(t_0) = \Phi ( \Phi^{-1}(x_0; t_0) ; t_0) = x_0$. 
  Thus $\eta$ is an admissible choice in Definition~\ref{def:evolving_hypersurfaces}, which yields
  \begin{equation*}
    V(x_0,t_0)
    = \eta^\prime(t_0) \cdot \nu(x_0,t_0)
    = ( \partial_t \Phi)( \Phi^{-1} (x_0;t_0); t_0 ) \cdot \nu(x_0,t_0).
  \end{equation*}
  Consequently, $V$ has the stated representation in terms of $\Phi$.
\end{proof}

By means of the transformation $\Phi(\cdot;t)\colon \overline{\Omega(0)} \rightarrow \overline{\Omega(t)}$, we may transform Lebesgue and Sobolev functions defined on $\Omega(t)$ to functions on $\Omega(0)$.
For this purpose, for $t\in [0,T]$, we introduce the transformation $\Phi_\ast(t)$ defined by
\begin{equation}\label{eq:def_Phi_ast_t}
  ( \Phi_\ast(t) f ) (\xi) = (\nabla \Phi)^{-1}( \Phi(\xi;t) ; t ) f ( \Phi(\xi;t) )
\end{equation}
for $\xi \in \overline{ \Omega(0) }$ and $f\colon \overline{ \Omega(t) }\rightarrow \R^n$; see \cite[equation~(10)]{saal}.
The main properties of the transformation \eqref{eq:def_Phi_ast_t} are collected in the next lemma; see also \cite[Section~3]{saal}.    
In particular, it turns out that $\Phi_\ast(t)$ defines a divergence-preserving operator.

\begin{lemma}[Properties of $\Phi_\ast(t)$]\label{lem:transformation_time_dep_domains}
    Suppose that $(\Omega(t))_{t\in[0,T]}$ is as in Assumptions~\ref{ass:reg_omega_t}.
    Let $k=0,1,2$, $l = 1,2$, $q\in [ 1,\infty ]$ and $t\in [0,T]$.
    Then the operator $\Phi_\ast(t)$ defined by \eqref{eq:def_Phi_ast_t} has the following properties.
    \begin{enumerate}
    \item
      The mapping $\Phi_\ast(t)\colon W^{k,q}(\Omega(t))^n \rightarrow W^{k,q}(\Omega(0))^n$ is an isomorphism.
      Its inverse operator $\Phi_\ast^{-1}(t)$ is given by $( \Phi_\ast^{-1}(t) h ) ( x ) = ( \nabla \Phi )( x ; t) h ( \Phi^{-1} (x;t) )$ for $h \in W^{k,q}(\Omega(0))^n$.
       
    \item
      There are constants $C_1, C_2 >0$, which do not depend on $t$, such that
      \begin{equation}\label{eq:phi_ast_continuity}
        C_1 \| \Phi_\ast(t) f \|_{W^{k,q}(\Omega(0))^n}
        \leq \| f \|_{W^{k,q}(\Omega(t))^n}
        \leq C_2 \| \Phi_\ast(t) f \|_{W^{k,q}(\Omega(0))^n}.
      \end{equation}
    
    \item
      The mapping $\Phi_\ast(t)\colon W^{l,q}_0(\Omega(t))^n \rightarrow W^{l,q}_0(\Omega(0))^n$ is an isomorphism.
  
    \item
      For any $f \in W^{1,q}(\Omega(t))^n$, there holds $\dive ( f ) \circ \Phi(\cdot;t) \in L^q( \Omega(0) )$ and $\dive ( \Phi_\ast(t) f ) = \dive ( f ) \circ  \Phi(\cdot;t)$.
      Moreover, $\Phi_\ast(t)\colon L^q_{\sigma}(\Omega(t)) \rightarrow L^q_{\sigma}(\Omega(0))$ is an isomorphism.
   \end{enumerate}
  \end{lemma}

  \begin{proof}
The proof is straightforward.
See \cite[Lemma~4.4.6]{daube} for details.
\end{proof}

We are interested in functions of the form $t \mapsto f(t) \in L^q( \Omega(t) )$ or $t \mapsto f(t) \in W^{k,q}( \Omega(t) )$.
To define these function spaces, we will always suppose that $(\Omega(t))_{t\in[0,T]}$ satisfies the regularity conditions gathered together in Assumptions~\ref{ass:reg_omega_t}.
For functions $f\in L^1_{\mathrm{loc}}(\Omega_T)$, the distributional derivatives $\partial_t^k \partial^\alpha f$ with $( k, \alpha) \in \N_0 \times \N^n_0$ are well-defined.
This allows us to define the following Bochner-type function spaces.
\begin{definition}[Lebesgue and Sobolev spaces on time-dependent domains]\label{def:Lebesgue_time_dep_domains}
  Suppose that $(\Omega(t))_{t\in[0,T]}$ satisfies Assumptions~\ref{ass:reg_omega_t}.
  Let $s,r \in [1,\infty]$ and $q \in \N_0$.
  \begin{enumerate}
  \item
    The space $L^s( 0,T; L^r(\Omega(t)) )$ consists of all $f\in L^1_{\mathrm{loc}}(\Omega_T)$ such that $f(t) \in L^r(\Omega(t))$ for a.e.\ $t\in (0,T)$, and $( t \mapsto \| f(t) \|_{L^r(\Omega(t))} ) \in L^s(0,T)$.
  
  \item
    The space $L^s( 0,T; W^{q,r}( \Omega(t)) )$ consists of all $f\in L^s( 0,T; L^r(\Omega(t)) )$ such that, for all $\alpha \in \N_0^n$ with $\alpha_1+\alpha_2+ \dots + \alpha_n \leq q$, there holds $\partial^\alpha f \in L^s( 0,T; L^r(\Omega(t)) )$.
    
  \item
    The space $W^{1,s}( 0,T; W^{q,r}( \Omega(t)) )$ consists of all $f\in L^s( 0,T; W^{q,r}( \Omega(t)) )$ such that $\partial_t f\in L^s( 0,T; W^{q,r}( \Omega(t)) )$.

   \item
    The vector-valued versions of the above spaces are given by
    \begin{equation*}
    \begin{split}
      L^s( 0,T; W^{q,r}( \Omega(t))^n ) &= L^s( 0,T; W^{q,r}( \Omega(t)) )^n,\\
      W^{1,s}( 0,T; W^{q,r}( \Omega(t))^n ) &= W^{1,s}( 0,T; W^{q,r}( \Omega(t)) )^n.
    \end{split}
    \end{equation*}
    
  \item
    Let $X(t)$ stand for either $W^{q,r}( \Omega(t) )$ or $W^{q,r}( \Omega(t) )^n$.
    The space $L^s(0,T ; X(t))$ is equipped with the norm
    \begin{equation*}
      \| f \|_{L^s(0,T ; X(t))}
      =
      \begin{cases}
        \left( \int_0^T \| f(t) \|_{X(t)}^s \,\mathrm{d}t \right)^{\scriptstyle \frac{1}{s}} & \text{if} \ s < \infty,\\
        \esssup_{t\in(0,T)}\| f(t) \|_{X(t)} & \text{if} \ s=\infty.
    \end{cases}
    \end{equation*}
    The space $W^{1,s}(0,T ; X(t))$ is equipped with the norm
    \begin{equation*}
      \| f \|_{W^{1,s}(0,T ; X(t))}
      =
      \left( \| f \|_{L^s(0,T ; X(t))}^2 + \| \partial_t f \|_{L^s(0,T ; X(t))}^2 \right)^{\scriptstyle \frac{1}{2}}.
    \end{equation*}
\end{enumerate}
\end{definition}

\begin{remark}
  We want to point out that, in the foregoing Definition~\ref{def:Lebesgue_time_dep_domains} we crucially used the fact that all defined function spaces are subspaces of $L^1( \Omega_T )$.
\end{remark}
   
We may use $\Phi_\ast(t)$ to transform functions from the previous definitions to functions taking values in time-independent Lebesgue or Sobolev spaces, i.e., functions belonging to the usual Bochner spaces.
To this end, we define $\Phi_\ast f$ by
\begin{equation}\label{eq:def_Phi_ast}
  t \mapsto \Phi_\ast(t)f(\cdot,t).
\end{equation}

Owing to the time-independent bounds on $\Phi_\ast(t)$ and its inverse $\Phi_\ast^{-1}(t)$  from Lemma~\ref{lem:transformation_time_dep_domains}, the transformation properties carry over to $\Phi_\ast$, as we now show.
The function spaces introduced in Definition~\ref{def:Lebesgue_time_dep_domains} are transformed as follows.

\begin{proposition}[Properties of $\Phi_\ast$]\label{prop:transformation_Phi_ast}
  Suppose that Assumptions~\ref{ass:reg_omega_t} hold true.
  Let $s\in [1,\infty]$, $q \in [1,\infty)$ and $k=0,1,2$.
  Denote by $X(\tau)$, $\tau \in [0,T]$, either of the spaces $W^{k,q}(\Omega(\tau))^n$, $W^{k,q}_0(\Omega(\tau))^n$ or $L^q_{\sigma}(\Omega(\tau))$.
  Then $\Phi_\ast$, given by \eqref{eq:def_Phi_ast}, is a diffeomorphism between the spaces $L^s( 0,T ; X(t) )$ and $L^s( 0,T ; X(0) )$ as well as between the spaces
  \begin{equation*}
    W^{1,s} ( 0,T ; L^q(\Omega(t))^n ) \cap L^s( 0,T ; W^{1,q}(\Omega(t))^n )
  \end{equation*}
  and
  \begin{equation*}
    W^{1,s} ( 0,T ; L^q(\Omega(0))^n )\cap L^s( 0,T ; W^{1,q}(\Omega(0))^n ).
  \end{equation*}
\end{proposition}

\begin{proof}
  By Lemma~\ref{lem:transformation_time_dep_domains}, $\Phi_\ast$ is an isomorphism between spaces of the form $L^s( 0,T ; X(t) )$ and $L^s( 0,T ; X(0) )$.
  For the proof of the remaining claim, we study the transformation of time derivatives.
  Let $f \in W^{1,s} ( 0,T ; L^q(\Omega(t))^n ) \cap L^s( 0,T ; W^{1,q}(\Omega(t))^n )$.
  Hence $\Phi_\ast f \in L^s ( 0,T ; W^{1,q}(\Omega(0))^n )$.
  By the definition of $W^{1,s} ( 0,T ; L^q(\Omega(t))^n )$, there holds that $\partial_t f \in L^s ( 0,T ; L^q(\Omega(t))^n )$.
  To prove that $\partial_t ( \Phi_\ast f )$ belongs to $L^s ( 0,T ; L^q(\Omega(0))^n )$, we use the mapping $\varLambda\colon ( \xi ,t ) \mapsto ( \Phi ( \xi  ; t ) , t )$, which belongs to $C^{3,1}_{\mathrm b} ( \Omega(0) \times (0,T) )^{n+1}$, by Corollary~\ref{cor:reg_inv_phi}, and that we may write
  \begin{equation}\label{eq:rep_Phi_ast_Lambda}
    \Phi_\ast g = ( (\nabla \Phi)^{-1} \circ \varLambda ) ( g \circ \varLambda )
  \end{equation}
  for any $g \in L^s ( 0,T ; L^q(\Omega(0))^n )$. 
  Using the product and the chain rule, we see
  \begin{equation*}
  \begin{split}
    \partial_t ( \Phi_\ast f )
    & =\partial_t \big( ( (\nabla \Phi)^{-1} \circ  \varLambda ) ( f \circ \varLambda ) \big)\\
    & =\partial_t \big( (\nabla \Phi)^{-1} \circ  \varLambda \big) ( f \circ \varLambda ) + \big( (\nabla \Phi)^{-1} \circ  \varLambda \big)  \partial_t ( f \circ \varLambda )\\
    & =\partial_t \big( (\nabla \Phi)^{-1} \circ  \varLambda \big) ( f \circ \varLambda ) + \big( (\nabla \Phi)^{-1} \circ  \varLambda \big) \big( (\nabla f \circ \varLambda ) \partial_t \Phi + \partial_t f \circ \varLambda \big)\\
    & =\partial_t \big( (\nabla \Phi)^{-1} \circ  \varLambda \big) ( f \circ \varLambda ) + \big( (\nabla \Phi)^{-1} \circ  \varLambda \big) \big( (\sum_{i=1}^n ( \partial_i f \circ \varLambda ) \partial_t \Phi_i ) + \partial_t f \circ \varLambda \big).
  \end{split}
  \end{equation*}
  
  Recalling \eqref{eq:rep_Phi_ast_Lambda}, it follows that
  \begin{equation*}
  \begin{split}
    \partial_t ( \Phi_\ast f )
     &=\partial_t \big( (\nabla \Phi)^{-1} \circ  \varLambda \big) ( \nabla \Phi \circ \varLambda ) \Phi_\ast f + \big( \sum_{i=1}^n \Phi_\ast ( \partial_i f )  \partial_t \Phi_i \big) + \Phi_\ast ( \partial_t f \big).
  \end{split}
  \end{equation*}
  Since $\Phi$ and $\varLambda$ belong to $C^{3,1}_{\mathrm b}( \overline{Q_0} )^n$ and $C^{3,1}_{\mathrm b}( \overline{Q_0} )^{n+1}$, respectively, the functions $\partial_t \big( (\nabla \Phi)^{-1} \circ  \varLambda \big)$, $\nabla \Phi \circ  \varLambda$ and $\partial_t \Phi$ are continuous and bounded on $Q_0 = \Omega(0) \times (0,T)$.
  Moreover, $\Phi_\ast f$, $\Phi_\ast( \partial_i f )$, $i=1,\dots,n$, and $\Phi_\ast( \partial_t f )$ belong to $L^s ( 0,T ; W^{1,q}(\Omega(0))^n )$.
  This implies $\partial_t ( \Phi_\ast f ) \in  L^s ( 0,T ; L^q(\Omega(0))^n )$, and thus $\Phi_\ast f  \in  W^{1,s} ( 0,T ; L^q(\Omega(0))^n )$.
  The remaining claim follows by similar arguments, as in the proof of Lemma~\ref{lem:transformation_time_dep_domains}.
\end{proof}

In the spirit of Theorem~\ref{thm:transport_theorem}, we obtain the following integration-by-parts formula for Sobolev spaces on time-dependent domains.

\begin{lemma}[Integration by parts]\label{lem:int_part_normal_velocity}
  Suppose that Assumptions~\ref{ass:reg_omega_t} hold true.
  For $r\in [1,\infty)$, let $f\in W^{1,r}( \Omega_T  )$ and $\varphi \in C_0^\infty( \Omega \times (0,T) )$.
  Then there holds
  \begin{equation}\label{eq:int_by_part_time_dep_dom}
    \int_0^T \!\int_{\Omega(t)} \partial_t f \varphi \,\mathrm{d} x \,\mathrm{d} t
    = - \int_0^T \!\int_{\Omega(t)} f \partial_t \varphi \,\mathrm{d} x  \,\mathrm{d} t
    - \int_0^T \!\int_{\partial \Omega(t)} V f \varphi \,\mathrm{d} \Hm^{n-1}(x) \,\mathrm{d} t.
  \end{equation}
\end{lemma}

\begin{proof}
  By Corollary~\ref{cor:reg_inv_phi}, the space-time domain $\Omega_T$ has a Lipschitz boundary.
  By density of  $C^\infty( \overline{\Omega_T} ) \cap W^{1,r}(\Omega_T)$ in
  \begin{equation*}
    W^{1,r}(\Omega_T) = W^{1,r}( 0,T ; L^{r}( \Omega(t)) ) \cap L^{r}( 0,T ;  W^{1,r}( \Omega(t)) ),
  \end{equation*}
  see \cite[p.~127, Theorem~3]{evans_gariepy_mt}, there exists an approximating sequence $(f_m)_{m\in\N} \subset C^\infty( \overline{\Omega_T} ) \cap W^{1,r}(\Omega_T)$ such that $f_m \rightarrow f$ in $W^{1,r}(\Omega_T)$ as $m\rightarrow \infty$.
  Using Theorem~\ref{thm:transport_theorem}, we obtain
  \begin{equation*}
  \begin{split}
    0
    &= \int_0^T \frac{\mathrm d}{\mathrm d t} \int_{\Omega(t)} f_m \varphi \,\mathrm{d} x \,\mathrm{d} t\\
    &= \int_0^T \!\int_{\Omega(t)} \partial_t (f_m \varphi) \,\mathrm{d} x \,\mathrm{d} t + \int_0^T \!\int_{\partial \Omega(t)} V f_m \varphi \,\mathrm{d} \Hm^{n-1}(x)\,\mathrm{d} t.
  \end{split}
  \end{equation*}
  Hence
  \begin{equation*}
  \begin{split}
     &\int_0^T \!\int_{\Omega(t)} \partial_t f_m \varphi \,\mathrm{d} x \,\mathrm{d} t\\
     ={}& - \int_0^T \!\int_{\Omega(t)} f_m \partial_t \varphi \,\mathrm{d} x \,\mathrm{d} t - \int_0^T \!\int_{\partial \Omega(t)} V f_m \varphi \,\mathrm{d} \Hm^{n-1}(x)\,\mathrm{d} t.
  \end{split}
  \end{equation*}
  In view of Corollary~\ref{cor:reg_inv_phi} and Proposition~\ref{prop:representation_normal velocity}, the normal velocity $V$ is bounded.
  By standard properties of the trace operator \cite[p.~133, Theorem~3]{evans_gariepy_mt}, we obtain \eqref{eq:int_by_part_time_dep_dom} by letting $m\rightarrow \infty$ in the final equation.
\end{proof}

\section{Consistency of the Weak Formulation}\label{sec:sim:consistency}

The notion of weak solutions for the sharp-interface model incorporates the variational formulation \eqref{eq:sim:variational_formulation}, using test functions from the space $C_0^\infty( [0,T) ; C_{0,\sigma}^\infty( \Omega ) )$.
Just as in the theory of the incompressible Navier--Stokes equations, this choice removes the pressure function from the weak formulation, cf.\ \cite[Definition~V.1.1.1]{sohr}.
Thus it is not clear that the test space $C_0^\infty( [0,T) ; C_{0,\sigma}^\infty( \Omega ) )$ is appropriate.
To justify this choice, we will prove that, under additional regularity assumptions given below, it is possible to reconstruct a pressure function from the weak formulation.
To this end, we will basically proceed in two steps.
Firstly, we will reconstruct an associated pressure function in the whole space-time domain $\Omega\times(0,T)$.
Secondly, we shall readjust the associated pressure function separately in the space-time domains
\begin{equation*}
  \Omega^-= \bigcup_{t\in(0,T)} {(\Omega^-(t) \times \{ t \})} \ \text{and} \ \Omega^+= \bigcup_{t\in(0,T)} {(\Omega^+(t) \times \{ t \})}
\end{equation*}
to satisfy the dynamical Young--Laplace law \eqref{eq:sim:pressure_jump} in an appropriate trace sense.

\begin{assumptions}\label{ass:consistency}
  Assume that $\Omega \subset \R^n$ is a bounded domain with boundary $\partial \Omega$ of class $C^3$.
  Let $(\rho, v)$ be a weak solution of the free-boundary problem \eqref{eq:sim:bulk1}--\eqref{eq:sim:ic} in the sense of Definition~\ref{def:sim:weak_solution} with respect to prescribed initial data $( \rho^{(i)}, v^{(i)} ) \in BV(\Omega,\{\beta_1,\beta_2\}) \times H^1_{0,\sigma}(\Omega)$, such that the measure-theoretic representative set $\Omega^{-,(i)} = \Omega^-(0)$ of $\rho^{(i)}$ is compactly contained in $\Omega$ and has a $C^3$-boundary.
  Moreover, let the following regularity properties hold true.
\begin{enumerate}
  \item \textbf{Regularity of interface.}
    For any $t\in [0,T]$, $\Phi^-(\cdot;t)\colon \overline{\Omega^-(0)} \rightarrow \overline{\Omega^-(t)}$ is a diffeomorphism as in Assumptions~\ref{ass:reg_omega_t}, such that the time evolution of the measure-theoretic representative set $\Omega^-(t)$ of $\rho(t)$ is described by $\Phi^-(\cdot;t)$, i.e., for every $t\in [0,T]$, there holds
    \begin{equation*}
      \Omega^-(t) = \big\{ \Phi^-(\xi;t) \,:\, \xi \in \Omega^-(0) \big\}
     \ \text{and} \
      \overline{ \Omega^-(t) } = \big\{ \Phi^-(\xi;t) \,:\, \xi \in \overline{ \Omega^-(0) } \big\}.
    \end{equation*}
    Additionally,  for all $t\in [0,T]$, the interface $\Gamma(t) = \partial ( \Omega^-(t) ) \cap \Omega$ is compactly contained in $\Omega$, that is, $\Gamma(t) = \partial ( \Omega^-(t) ) \subset \subset \Omega$.
    Denote by $\nu^-= \nu^-( \cdot , t)$ the unit normal to $\Gamma(t)$ pointing outward to $\Omega^-(t)$ and by $V=V(\cdot,t)$ the normal velocity of $( \Gamma (t))_{t\in [ 0 , T ]}$ with respect to $\nu^-$. 
    Similarly, let the time evolution of $\Omega^+(t) = \Omega \setminus ( \Omega^-(t) \cup \Gamma(t) )$ be described by a diffeomorphism $\Phi^+(\cdot;t)\colon \overline{\Omega^+(0)} \rightarrow \overline{\Omega^+(t)}$ satisfying Assumptions~\ref{ass:reg_omega_t}.
  
  \item \textbf{Regularity of velocity.}
    \begin{equation*}
      \left. v \right|_{\Omega^\pm} \in L^2( 0,T ; W^{2,2}(\Omega^\pm(t))^n ) \cap W^{1,2}( 0,T ; L^{2}(\Omega^\pm(t))^n ).
    \end{equation*}
\end{enumerate}
\end{assumptions}

\subsection{The Mean-Curvature Functional for Smooth Interfaces}

Due to Assumptions~\ref{ass:consistency}, the family of interfaces $( \Gamma(t) )_{t\in [0,T]}$ has additional regularity properties.
This allows us to extend the mean-curvature function to the space-time domain $\Omega \times (0,T)$.
For the proof, we study the transformation of the trace spaces $L^2(\Gamma(t)) = L^2(\partial ( \Omega^-(t) ) )$ and $H^{\scriptstyle \frac{1}{2}}(\Gamma(t)) = H^{\scriptstyle \frac{1}{2}}(\partial ( \Omega^-(t) ) )$.

\begin{lemma}[Transformation of trace spaces]\label{lem:reg_pullback}
  Suppose that Assumptions~\ref{ass:consistency} hold true, and let $t\in [0,T]$.
  Then the pullback operator $\Phi^-_{-t}$, defined by $\Phi^-_{-t} u = u \circ \Phi^-(\cdot,t)$ for $u \in L^2( \Gamma(t) )$, induces linear homeomorphisms
  \begin{equation*}
    \Phi^-_{-t}\colon L^2(\Gamma(t)) \rightarrow L^2(\Gamma(0))
    \ \text{and} \
    \Phi^-_{-t}\colon H^{\scriptstyle \frac{1}{2}}(\Gamma(t)) \rightarrow H^{\scriptstyle \frac{1}{2}}(\Gamma(0)),
  \end{equation*}
  such that
  \begin{equation} \label{eq:est_L2_interface}
    C_1 \| u \|_{L^2(\Gamma(t))}
    \leq \| \Phi^-_{-t} u  \|_{L^2(\Gamma(0))}
    \leq C_2 \| u \|_{L^2(\Gamma(t))}
  \end{equation}
  for every $u \in L^2(\Gamma(t))$, and
  \begin{equation*} 
    C_1 \| u \|_{H^{\scriptstyle \frac{1}{2}}(\Gamma(t))}
    \leq \| \Phi^-_{-t} u  \|_{H^{\scriptstyle \frac{1}{2}}(\Gamma(0))}
    \leq C_2 \| u \|_{H^{\scriptstyle \frac{1}{2}}(\Gamma(t))}
  \end{equation*}
  for every $u \in H^{\scriptstyle \frac{1}{2}}(\Gamma(t))$ with constants $C_1, C_2 > 0$ independent of $u$ and $t$.
  In particular, there are constants $C_3, C_4 > 0$ such that
  \begin{equation}\label{eq:est_Hausdorff_measure_interface}
    C_3 \Hm^{n-1} (\Gamma(t)) )
    \leq \Hm^{n-1} (\Gamma(0)) )
    \leq C_4 \Hm^{n-1} (\Gamma(t)) ).
  \end{equation}
\end{lemma}

\begin{proof}
  The estimate \eqref{eq:est_Hausdorff_measure_interface} follows from \eqref{eq:est_L2_interface} applied to the constant function $u \equiv 1$.
  The proof of the remaining claims can be found in \cite[Section~5.4.1]{alphonse_linear_parabolic}.
\end{proof}

\begin{lemma}[Mean-curvature functional]\label{lem:extension_mean_curvature}
  If Assumptions~\ref{ass:consistency} hold true, then there exists a function $m \in L^\infty( 0,T ; H^1(\Omega^-(t)) )$ with the following properties.
  \begin{enumerate}
    \item
      Let $t \in [ 0,T ]$.
      For the trace of $m(t)$ on the boundary $\Gamma(t) = \partial \Omega^-(t)$, there holds
      \begin{equation}\label{eq:trace_m_kappa}
        \left. m(t) \right|_{\Gamma(t)} = \kappa(t).
      \end{equation}
      
    \item 
    The zero extension $K$ of $\nabla m$ to $\Omega \times (0,T)$ belongs to $L^\infty( 0,T ; L^2(\Omega)^n )$ and, for every $\psi \in C_0^\infty( [0,T) ; C_{0,\sigma}^\infty( \Omega ) )$, there holds
    \begin{equation}\label{eq:regular_curvature_functional}
    \begin{split}
      \int_0^T \!\int_{\Omega} K \cdot \psi \,\mathrm{d} x \,\mathrm{d} t
        &= \int_0^T \!\int_{\Gamma(t)} \kappa \nu^- \cdot \psi \,\mathrm{d}\Hm^{n-1}(x) \,\mathrm{d} t\\
        &= \int_0^T \!\int_{\Gamma(t)} \nu^- \otimes \nu^- : \nabla\psi \,\mathrm{d}\Hm^{n-1}(x) \,\mathrm{d} t.
    \end{split}
    \end{equation} 
    
  \end{enumerate}
\end{lemma}

\begin{proof}
  Let $t\in [0,T]$.
  We apply the pullback operator $\Phi^-_{-t}\colon H^{\scriptstyle \frac{1}{2}}(\Gamma(t)) \rightarrow H^{\scriptstyle \frac{1}{2}}(\Gamma(0))$, introduced in Lemma~\ref{lem:reg_pullback}, to the mean-curvature function $\kappa(t) \in C^1(\Gamma(t)) \subset H^{\scriptstyle \frac{1}{2}}(\Gamma(t))$.
  For notational convenience, we suppress the upper index $^-$ and simply write $\Phi = \Phi^-$ and $\Phi_{-t} = \Phi^-_{-t}$ in the remainder of this proof.
  We define $\tilde \kappa (x,t) = \Phi_{-t} \kappa ( x, t) = \kappa ( \Phi(x;t) , t)$ for $x\in \Gamma(0) = \partial ( \Omega^-(0) )$.
  Since $\Omega^-(0)$ has a $C^3$-boundary, there exists a weak solution $\tilde u = \tilde u(t) \in H^1(\Omega^-(0))$ of
 \begin{equation}\label{eq:Dirichlet_mean_curvature}
  \begin{aligned}
    \Delta \tilde u (t) &= 0 &&\text{in} \ \Omega^-(0),\\
    \tilde u(t) &= \tilde \kappa(t) &&\text{on} \ \Gamma(0).
  \end{aligned}
  \end{equation}
  depending on $t$, which additionally satisfies the estimate
  \begin{equation*}
    \| \tilde u(t) \|_{H^1(\Omega^-(0))}
    \leq C( \Omega^-(0) ) \| \tilde \kappa(t) \|_{H^{\scriptstyle \frac{1}{2}}(\Gamma(0))},
  \end{equation*}
  for some constant $C( \Omega^-(0) )>0$, depending on $\Omega^-(0)$, but independent of $t$; see~\cite[Theorem~III.4.1]{boyer_fabrie}.
  By Assumptions~\ref{ass:consistency} and the foregoing Lemma~\ref{lem:reg_pullback}, we infer that for a suitable constant $C>0$, independent of $t$, there holds
  \begin{equation*}
    \| \tilde u(t) \|_{H^1(\Omega^-(0))}
    \leq C( \Omega^-(0) ) \| \kappa(t) \|_{H^{\scriptstyle \frac{1}{2}}(\Gamma(t))}
    \leq C.
  \end{equation*}
  Therefore, the function $m\colon \Omega^- \rightarrow \R$ defined by $m(x,t)  = \tilde u( \Phi^{-1}(x;t) , t)$ for $x \in \Omega^-(t)$ belongs to $L^\infty( 0,T ; H^1(\Omega^-(t)) )$ and, by construction, satisfies \eqref{eq:trace_m_kappa}.
  \par
  Concerning the second claim, we define $K\colon \Omega \times (0,T) \rightarrow \R$, for any $(x,t) \in \Omega \times (0,T)$, by
  \begin{equation*}
    K(x,t) =
    \begin{cases}
      \nabla m(x,t) &\ \text{if} \ x \in \Omega^-(t),\\
      0 &\ \text{if} \ x \in \Omega \setminus \Omega^-(t).
    \end{cases}
  \end{equation*}
  Then $K$ belongs to $L^\infty (0,T ; L^2 (\Omega)^n)$.
  For every $\psi \in C_0^\infty( [0,T) ; C_{0,\sigma}^\infty( \Omega ) )$, we then obtain
  \begin{equation*}
  \begin{split}
    \int_0^T \!\int_{\Omega} K(t) \cdot \psi(t)  \,\mathrm{d}x \,\mathrm{d} t
    ={}&\int_0^T \!\int_{\Gamma(t)} m(t) \psi(t) \cdot \nu^-(t) \,\mathrm{d}\Hm^{n-1}(x) \,\mathrm{d} t.
  \end{split}
  \end{equation*}
  Taking into account \eqref{eq:trace_m_kappa}, we conclude the first identity in \eqref{eq:regular_curvature_functional}.
  Noting that the last equality in \eqref{eq:regular_curvature_functional} follows from Lemma~\ref{lem:weak_curvature} finishes the proof.
  \end{proof}

  For our purposes, it is important to note that, if Assumptions~\ref{ass:consistency} are satisfied, then Lemma~\ref{lem:extension_mean_curvature} allows one to replace \eqref{eq:sim:variational_formulation} by
  \begin{equation}\label{eq:sim:reg_var_formulation}
  \begin{split}
      &\int_0^T \!\int_{\Omega} \rho v \cdot \partial_t\psi + \rho v \otimes v:\nabla \psi  - 2 \mu( \rho ) D v : D \psi \,\mathrm{d}x \,\mathrm{d}t
      + \int_{\Omega} \rho^{(i)} v^{(i)} \cdot \psi(0)\,\mathrm{d}x\\
     ={}&- 2 \stc \int_0^T \!\int_{\Omega} K \cdot \psi \,\mathrm{d} x \,\mathrm{d} t
  \end{split}
  \end{equation} 
  for all $\psi \in C_0^\infty( [0,T) ; C_{0,\sigma}^\infty(\Omega) )$.
  It is also convenient to introduce $\mathcal{G} \in \mathcal{D}^\prime(\Omega \times (0,T))^n$ given by
  \begin{equation}\label{eq:sim:def_G}
  \begin{split}
    \langle \mathcal{G} , \psi \rangle_{\mathcal{D}(\Omega \times (0,T))^n}
    ={}&\int_0^T \!\int_{\Omega} \rho v \cdot \partial_t\psi + \rho v \otimes v:\nabla \psi  - 2 \mu( \rho ) D v : D \psi \,\mathrm{d}x \,\mathrm{d}t\\
      &+ 2 \stc \int_0^T \!\int_{\Omega} K \cdot \psi \,\mathrm{d} x \,\mathrm{d} t
  \end{split}
  \end{equation} 
  for $\psi \in C_0^\infty( \Omega \times (0,T) )^n$.
  Note that,
for all $\psi \in C_0^\infty( (0,T) ; C_{0,\sigma}^\infty(\Omega) )$, there holds
  \begin{equation}\label{eq:sim:G_vanishes}
    \langle \mathcal{G} , \psi \rangle_{\mathcal{D}(\Omega \times (0,T))^n}
    = 0.
  \end{equation}

\subsection{Existence of an Associated Pressure Function}

We shall prove the existence of an associated pressure function, that is, a distribution $p \in \mathcal{D}^\prime(\Omega \times (0,T))$ such that
\begin{equation*}
  \nabla p
  =
  - \rho \partial_t v - \dive( \rho v \otimes v ) - 2 \mu( \rho ) \dive( Dv ) + 2\stc K
  \ \text{in} \ \mathcal{D}^\prime(\Omega \times (0,T))^n.
\end{equation*}

The theory of the incompressible Navier--Stokes equations provides us with the following key tool.

\begin{theorem}\label{thm:reconstruction_pressure_abstract}
  Let $r,s \in (1,\infty)$ and let $r^\prime = \tfrac{r}{r-1}$.
  If $\mathcal{F} \in L^s ( 0,T ; W^{-1,r}(\Omega)^n)$ satisfies
  \begin{equation*}
    \int_0^T \langle \mathcal{F}(t) , \psi(t) \rangle_{W^{1,r^\prime}_0(\Omega)^n} \,\mathrm{d} t = 0
    \ \text{for all} \ \psi \in C_0^\infty( (0,T) ; C_{0,\sigma}^\infty( \Omega ) ), 
  \end{equation*}
  then there exists a unique $p \in L^s ( 0,T ; L^r(\Omega) )$ satisfying
  $\mathcal{F} = \nabla p$ in $\mathcal{D}^\prime(\Omega \times (0,T))^n$; that is,
  \begin{equation*}
    \langle \mathcal{F} , \psi \rangle_{\mathcal{D}(\Omega \times (0,T))^n}
    = \int_0^T  \langle \nabla p(t) , \psi ( t ) \rangle_{\mathcal{D}(\Omega)^n} \,\mathrm{d} t
    = - \int_0^T \!\int_\Omega p \dive(\psi) \,\mathrm{d} x \,\mathrm{d} t
  \end{equation*}
  for all $\psi \in C_0^\infty( \Omega \times (0,T) )^n$, and, for a.e.\ $t \in (0,T)$, there holds
  \begin{equation*}
    \int_\Omega p(t) \,\mathrm{d} x = 0.
  \end{equation*}
\end{theorem}

\begin{proof}
  See~\cite[Lemma~IV.1.4.1]{sohr}.
\end{proof}

Although the functional $\mathcal{G}$, defined by \eqref{eq:sim:def_G}, vanishes on $C_0^\infty( (0,T) ; C_{0,\sigma}^\infty( \Omega ) )$ due to \eqref{eq:sim:G_vanishes}, as the functional
\begin{equation*}
  \psi \mapsto \int_0^T \!\int_\Omega \rho v \cdot \partial_t \psi \,\mathrm{d} x \,\mathrm{d} t
\end{equation*}
does not in general belong to any $L^s ( 0,T ; W^{-1,r}(\Omega) )$-space.
To circumvent this problem, we improve the properties of this functional by taking into account Assumptions~\ref{ass:consistency}.

\begin{proposition}\label{prop:cons:time_derivatives_part_int}
  Suppose that Assumptions~\ref{ass:consistency} are satisfied.
  Then there holds
  \begin{equation}\label{eq:integration_time_der}
  \begin{split}
    &\int_0^T \!\int_{\Omega \setminus \Gamma(t)} \rho  \partial_t v \cdot  \psi \,\mathrm{d}x \,\mathrm{d}t\\
    ={}&- \int_0^T \!\int_{\Omega}  \rho  v \cdot \partial_t \psi \,\mathrm{d}x \,\mathrm{d}t
     - (\beta_1 - \beta_2)  \int_0^T \!\int_{\Gamma(t)}  V ( v \cdot \psi ) \,\mathrm{d}\Hm^{n-1}(x) \,\mathrm{d}t   
  \end{split}
  \end{equation}
  for any $\psi \in C^\infty_0( \Omega \times (0,T) )^n$. 
\end{proposition}

\begin{proof}
  Since, by Assumptions~\ref{ass:consistency}, $v$ belongs to
  \begin{equation*}
    L^2( 0,T ; H^1(\Omega^\pm(t))^n ) \cap W^{1,2}( 0,T ; L^2(\Omega^\pm(t))^n ),
  \end{equation*}
  for any $\psi \in C^\infty_0( \Omega \times (0,T) )^n$, the integration-by-parts formula (Lemma~\ref{lem:int_part_normal_velocity}) yields
  \begin{equation*}
  \begin{split}
    &\beta_1 \int_0^T \!\int_{\Omega^-(t)} \partial_t v \cdot \psi \,\mathrm{d} x \,\mathrm{d} t
    + \beta_2 \int_0^T \!\int_{\Omega^+(t)} \partial_t v \cdot \psi \,\mathrm{d} x \,\mathrm{d} t\\
    ={}& - \beta_1 \int_0^T \!\int_{\Omega^-(t)} v \cdot \partial_t \psi \,\mathrm{d} x  \,\mathrm{d} t
    - \beta_1 \int_0^T \!\int_{\Gamma(t)} V ( v \cdot \psi ) \,\mathrm{d} \Hm^{n-1}(x) \,\mathrm{d} t\\
     &- \beta_2 \int_0^T \!\int_{\Omega^+(t)} v \cdot \partial_t \psi \,\mathrm{d} x  \,\mathrm{d} t
    + \beta_2 \int_0^T \!\int_{\Gamma(t)} V ( v \cdot \psi ) \,\mathrm{d} \Hm^{n-1}(x) \,\mathrm{d} t.
  \end{split}
  \end{equation*}
  Recalling that $\rho = ( \beta_1 - \beta_2 ) \chi + \beta_2$ finally yields the claim.
\end{proof}

\begin{remark}[Time derivatives across the interface]
  In \eqref{eq:integration_time_der}, the domain of integration is $\Omega \setminus \Gamma(t) = \Omega^-(t) \cup \Omega^+(t)$ instead of the whole domain $\Omega$, despite the fact that $\Gamma(t)$ has Lebesgue measure zero.
  This is because, by Assumptions~\ref{ass:consistency}, the restrictions of $v$ to $\Omega^\pm$ belong to some $W^{1,q}( 0,T ; L^q( \Omega^\pm(t))^n)$-space.
  However, this does not give any information about the behaviour of $\partial_t v$ on the interface $\Gamma(t)$.
  In particular, we cannot assume that $\partial_t v$ exists in the sense of weak derivatives on $\Omega \times (0,T)$. 
\end{remark}

We now prove some preparatory results, which incorporate the additional properties from Assumptions~\ref{ass:consistency}, before we reconstruct the pressure function with the help of Theorem~\ref{thm:reconstruction_pressure_abstract}.

\begin{proposition}\label{prop:cons:div_free_normal_velocity}
  If Assumptions~\ref{ass:consistency} are satisfied, then $v$ has the following properties.
  \begin{enumerate}
  \item
    $\dive(v) = 0$ in $\Omega \times (0,T)$.
  
  \item
    For every $\varphi \in C^\infty_0(\Omega \times (0,T))$, there holds
    \begin{equation}\label{eq:weak_interface_velocity}
    \begin{split}
      \int_0^T \!\int_{\Gamma(t)} V \varphi \,\mathrm{d} \Hm^{n-1}(x) \,\mathrm{d} t
      &= - \int_0^T \!\int_{\Omega} \chi \partial_t \varphi  \,\mathrm{d} x \,\mathrm{d} t\\
      &= \int_0^T \!\int_{\Gamma(t)} (v\cdot \nu^-) \varphi \,\mathrm{d} \Hm^{n-1}(x) \,\mathrm{d} t.
    \end{split}
    \end{equation}
  
  \item
    For every $\psi \in C^\infty_0(\Omega \times (0,T))^n$, there holds
    \begin{equation}\label{eq:div_free_normal_velocity_v_weak}
      \int_0^T \!\int_{\Gamma(t)} V ( v \cdot \psi ) \,\mathrm{d} \Hm^{n-1}(x) \,\mathrm{d} t
      = \int_0^T \!\int_{\Gamma(t)} ( v \cdot \psi ) (v\cdot \nu^-) \,\mathrm{d} \Hm^{n-1}(x) \,\mathrm{d} t.
    \end{equation}
  \end{enumerate}
\end{proposition}

\begin{proof}
\begin{enumerate}
  \item
    By Definition~\ref{def:sim:weak_solution}, $v \in L^\infty( 0,T ; L_\sigma^2(\Omega) ) \cap L^2( 0,T ; H^1_0(\Omega)^n )$.
    In particular, this means that $v \in L^2( 0,T ; H^1_{0,\sigma}(\Omega)^n )$.
    Finally, Lemma~\ref{lem:sigma_characterisations} implies the first claim.

  \item
    Let $\varphi \in C^\infty_0(\Omega \times (0,T))$.
    The first equality in \eqref{eq:weak_interface_velocity} follows from the first statement of Theorem~\ref{thm:transport_theorem}.
    For the proof of the second equality in \eqref{eq:weak_interface_velocity}, we use that $\chi$ is a weak solution of the transport equation \eqref{eq:sim:transport_equation}.
    Thus, by \eqref{eq:transport_var}, we have
    \begin{equation*}
      - \int_0^T \!\int_{\Omega} \chi  \partial_t \varphi \,\mathrm{d} x \,\mathrm{d} t
      = \int_0^T \!\int_{\Omega} \chi v \cdot \nabla \varphi \,\mathrm{d} x \,\mathrm{d} t
      = \int_0^T \!\int_{\Omega^-(t)} v \cdot \nabla \varphi \,\mathrm{d} x \,\mathrm{d} t.
    \end{equation*}
    Using integration by parts and $\dive(v) = 0$ in $\Omega \times (0,T)$, it follows 
    \begin{equation*}
    \begin{split}
      &- \int_0^T \!\int_{\Omega} \chi  \partial_t \varphi \,\mathrm{d} x \,\mathrm{d} t\\
      ={}&  - \int_0^T \!\int_{\Omega^-(t)} \dive(v) \varphi  \,\mathrm{d} x \,\mathrm{d} t
        + \int_0^T \!\int_{\Gamma(t)} (v\cdot \nu^-) \varphi \,\mathrm{d} \Hm^{n-1}(x) \,\mathrm{d} t\\
      ={}& \int_0^T \!\int_{\Gamma(t)} (v\cdot \nu^-) \varphi \,\mathrm{d} \Hm^{n-1}(x) \,\mathrm{d} t.
    \end{split}
    \end{equation*}
    This proves \eqref{eq:weak_interface_velocity}.
  \item
    Due to Assumptions~\ref{ass:consistency}, there holds
    \begin{equation*}
      v \in W^{1,2}( 0,T ; L^{2}( \Omega^-(t))^n ) \cap L^{2}( 0,T ;  H^1( \Omega^-(t)) )^n ) = H^1(\Omega^-_T)^n.
    \end{equation*}
    By Corollary~\ref{cor:reg_inv_phi}, $\Omega^-_T$ has a Lipschitz boundary, and therefore $C^\infty\left(\overline{\Omega^-_T} \right)\cap H^1(\Omega^-_T)$ is dense in $H^1(\Omega^-_T)$; see \cite[p.~127, Theorem~3]{evans_gariepy_mt}.
    This means that there exists an approximating sequence $(v_m)_{m\in\N} \subset C^\infty\left(\overline{\Omega^-_T}\right)^n\cap H^1(\Omega^-_T)^n$ such that $v_m \rightarrow v$ in $H^1(\Omega^-_T)^n$ as $m\rightarrow \infty$.
    For any $\psi \in C^\infty_0(\Omega \times (0,T))^n$, for $m\rightarrow \infty$, we obtain
    \begin{equation}\label{eq:conv_time_der}
      \chi \partial_t ( v_m \cdot \psi ) \rightarrow \chi \partial_t ( v \cdot \psi ) \ \text{in} \ L^2(\Omega^-_T)
    \end{equation}
    and
    \begin{equation}\label{eq:conv_gradient}
      \chi \nabla ( v_m \cdot \psi ) \rightarrow \chi \nabla ( v \cdot \psi ) \ \text{in} \ L^2(\Omega^-_T)^n
    \end{equation}
    since $\chi \in L^\infty( \Omega \times (0,T))$ and $v \in L^2( 0,T; L^2( \Omega )^n )$.
    As $\chi$ is a weak solution of the transport equation, by \eqref{eq:transport_var}, we infer that
    \begin{equation*}
      \int_0^T \!\int_{\Omega} \chi \partial_t ( v_m \cdot \psi ) \,\mathrm{d} x \,\mathrm{d} t
      = - \int_0^T \!\int_{\Omega} \chi v \cdot \nabla ( v_m \cdot \psi )  \,\mathrm{d} x \,\mathrm{d} t.
    \end{equation*}
    Now \eqref{eq:conv_time_der} and \eqref{eq:conv_gradient} allow us to pass to the limit $m\rightarrow \infty$.
    This yields
    \begin{equation*}
    \begin{split}
     \int_0^T \!\int_{\Omega^-(t)} \partial_t ( v \cdot \psi ) \,\mathrm{d} x \,\mathrm{d} t
      &=\int_0^T \!\int_{\Omega} \chi \partial_t ( v \cdot \psi ) \,\mathrm{d} x \,\mathrm{d} t\\
      &= - \int_0^T \!\int_{\Omega} \chi v \cdot \nabla ( v \cdot \psi )  \,\mathrm{d} x \,\mathrm{d} t.
    \end{split}
    \end{equation*}
    As the integration-by-parts formula (see Lemma~\ref{lem:int_part_normal_velocity}) applies to the left-hand side and as $\dive(v) = 0$ in $\Omega^-$, we obtain
    \begin{equation*}
    \begin{split}
      \int_0^T \!\int_{\Gamma(t)} V ( v \cdot \psi ) \,\mathrm{d} \Hm^{n-1}(x) \,\mathrm{d} t
      &= \int_0^T \!\int_{\Omega^-(t)} \dive ( ( v \cdot \psi )  v ) \,\mathrm{d} x \,\mathrm{d} t\\
      &= \int_0^T \!\int_{\Gamma(t)} ( v \cdot \psi ) (v\cdot \nu^-) \,\mathrm{d} \Hm^{n-1}(x) \,\mathrm{d} t.
    \end{split}
    \end{equation*}  
    This justifies \eqref{eq:div_free_normal_velocity_v_weak}, which completes the proof.
\end{enumerate}
\end{proof}

Next, we explore the regularity of the convective term $(v \cdot \nabla) v$.

\begin{lemma}[Regularity of convective term]\label{lem:reg_convective_term}
  Let Assumptions~\ref{ass:consistency} be satisfied.
  Then $(v \cdot \nabla) v$ belongs to $L^2(0,T ; L^3(\Omega^\pm(t))^n )$.
\end{lemma}

\begin{proof}
  Let $\Phi_\ast$ be given by \eqref{eq:def_Phi_ast}.
  In view of Proposition~\ref{prop:transformation_Phi_ast}, there holds that
  \begin{equation}\label{eq:def_w}
    w
    = \Phi_\ast v
    \in L^2( 0,T ; W^{2,2}(\Omega^\pm(0))^n) \cap W^{1,2}( 0,T ; L^{2}(\Omega^\pm(0))^n),
  \end{equation}
  and it is sufficient to verify that $w \cdot \nabla w_i \in L^2(0,T ; L^3(\Omega^\pm(0)) )$ for $i=1,\dots,n$.
  To this end, we will use the continuous embedding
  \begin{equation}\label{eq:embedding_cont_pm}
  \begin{split}
    L^2&( 0,T ; H^1(\Omega^\pm(0)) ) \cap W^{1,2}( 0,T ; W^{-1,2}(\Omega^\pm(0)) )\\
    &\hookrightarrow C^0 ( [0,T]; L^2 ( \Omega^\pm(0) ) );
  \end{split}
  \end{equation}
  see \cite[Chapter~III, Theorem~4.10.2]{amann} and \cite[Th\'eor\`eme~12.4]{lions_magenes}.
  As \eqref{eq:def_w} implies that $w, \nabla w_i \in L^2( 0,T ; H^1(\Omega^\pm(0))^n) \cap W^{1,2}( 0,T ; W^{-1,2}(\Omega^\pm(0))^n)$, taking into account the embedding \eqref{eq:embedding_cont_pm}, we conclude that
  \begin{equation*}
    w
    \in C^0 ( [0,T]; H^1 ( \Omega^\pm(0) )^n )
    \hookrightarrow L^\infty( 0,T ; L^6(\Omega^\pm(0))^n).
  \end{equation*}
  Then, by H\"older's inequality, we obtain
  \begin{equation*}
    \| w \cdot \nabla w_i \|_{L^2( 0,T ; L^3( \Omega^\pm(0) ) )}
    \leq \| w \|_{L^\infty( 0,T ; L^6( \Omega^\pm(0) )^n )} \| \nabla w_i \|_{L^2 ( 0,T ; L^6( \Omega^\pm(0) )^n )},
  \end{equation*}
  which completes the proof.
\end{proof}

Using Proposition~\ref{prop:cons:time_derivatives_part_int}, we improve the regularity of the functional $\mathcal{G}$; see \eqref{eq:sim:def_G}.

\begin{proposition}\label{prop:reg_G_reg}
  Suppose that Assumptions~\ref{ass:consistency} hold true and let $\mathcal{G}$ be as in \eqref{eq:sim:def_G}.
  For $\psi \in C^\infty_0( \Omega \times (0,T) )^n$, define $\mathcal{G}_{\mathrm{reg}}\in \mathcal{D}^\prime(\Omega \times (0,T))^n$ by
  \begin{equation}\label{eq:def_G_reg}
  \begin{split}
    &\langle \mathcal{G}_{\mathrm{reg}} , \psi \rangle_{\mathcal{D}(\Omega \times (0,T))^n}\\
    ={}&- \int_0^T \!\int_{\Omega \setminus \Gamma(t)} \rho \partial_t v \cdot \psi \,\mathrm{d}x \,\mathrm{d}t - \int_0^T \!\int_{\Omega} \rho v \cdot \nabla v \cdot \psi \,\mathrm{d}x \,\mathrm{d}t\\
    &- 2 \int_0^T \!\int_{\Omega}\mu( \rho ) D v : D \psi \,\mathrm{d}x \,\mathrm{d}t
    + 2 \stc \int_0^T \!\int_{\Omega} K \cdot \psi \,\mathrm{d} x \,\mathrm{d} t.
  \end{split}
  \end{equation}
  Then $\mathcal{G}_{\mathrm{reg}}$ extends to a functional belonging to $L^2( 0,T ; H^{-1}( \Omega )^n )$.
  Moreover, there holds $\langle \mathcal{G} , \psi \rangle_{\mathcal{D}(\Omega \times (0,T))^n} = \langle \mathcal{G}_{\mathrm{reg}} , \psi \rangle_{\mathcal{D}(\Omega \times (0,T))^n}$ for all $\psi \in C^\infty_0( \Omega \times (0,T) )^n$.
\end{proposition}

\begin{proof}
  In view of Assumptions~\ref{ass:consistency}, Lemma~\ref{lem:reg_convective_term}, Definition~\ref{def:sim:weak_solution} and Lemma~\ref{lem:extension_mean_curvature}, $\mathcal{G}_{\mathrm{reg}}$ extends to a functional belonging to the class $L^2( 0,T ; H^{-1}( \Omega )^n )$.
  \par
  Let $\psi \in C^\infty_0( \Omega \times (0,T) )^n$.
  It suffices to show that
  \begin{equation}\label{eq:time_dev_reg}
  \begin{split}
    &\int_0^T \!\int_{\Omega}  \rho  ( v \cdot \partial_t \psi + v \otimes v : \nabla \psi ) \,\mathrm{d}x \,\mathrm{d}t\\
    ={}& - \int_0^T \!\int_{\Omega \setminus \Gamma(t)} \rho  \partial_t v \cdot  \psi \,\mathrm{d}x \,\mathrm{d}t
    - \int_0^T \!\int_{\Omega} \rho \dive( v \otimes v ) \cdot \psi \,\mathrm{d}x \,\mathrm{d}t.
  \end{split}
  \end{equation}
  To this end, we integrate by parts on $\Omega^\pm(t)$, and use Proposition~\ref{prop:cons:div_free_normal_velocity}, to see that
  \begin{equation*}
  \begin{split}
    &\int_0^T \!\int_{\Omega} \rho \dive \left( (v\otimes v)  \psi \right) \,\mathrm{d}x \,\mathrm{d}t\\
    ={}& ( \beta_1 - \beta_2) \int_0^T \!\int_{\Gamma(t)}  (v\cdot \nu^-) (v \cdot \psi) \,\mathrm{d}\Hm^{n-1}(x) \,\mathrm{d}t\\
    ={}&(\beta_1 - \beta_2) \int_0^T \!\int_{\Gamma(t)}  V ( v \cdot \psi ) \,\mathrm{d}\Hm^{n-1}(x) \,\mathrm{d}t.
  \end{split}
  \end{equation*}
%
  Finally, applying Proposition~\ref{prop:cons:time_derivatives_part_int} implies \eqref{eq:time_dev_reg}.
\end{proof}

Taking into account the additional smoothness Assumptions~\ref{ass:consistency}, we can prove the existence of an associated pressure function.
 
\begin{theorem}[Reconstruction of associated pressure]\label{thm:reconstruction_pressure}
  Let Assumptions~\ref{ass:consistency} be satisfied.
  Then there exists some function $p \in L^2( 0,T ; L^2 ( \Omega ) )$ such that its restrictions $p^\pm = \left. p \right|_{\Omega^\pm}$ to $\Omega^\pm$ belong to $L^2( 0,T ; H^1(\Omega^\pm(t)))$ and satisfy
  \begin{equation}\label{eq:nabla_p_minus}
  \begin{split}
    \nabla p^- &= - \beta_1 \partial_t v + \mu ( \beta_1 ) \Delta v  - \beta_1 (v \cdot \nabla) v - 2 \stc K \ \text{a.e.\ in} \ \Omega^-,\\
    \nabla p^+ &= - \beta_2 \partial_t v + \mu ( \beta_2 ) \Delta v  - \beta_2 (v \cdot \nabla) v - 2 \stc K \ \text{a.e.\ in} \ \Omega^+,
  \end{split}
  \end{equation}
  where $K \in L^2( 0,T ; L^2(\Omega)^n )$ denotes the extension of the mean-curvature function as in Lemma~\ref{lem:extension_mean_curvature}. 
\end{theorem}

\begin{proof}
  In view of Proposition~\ref{prop:reg_G_reg} and \eqref{eq:sim:G_vanishes}, for any $\psi \in C^\infty_0( (0,T) ; C^\infty_{0,\sigma}(\Omega) )$, there holds
  $
    \langle \mathcal{G}_{\mathrm{reg}} , \psi \rangle_{\mathcal{D}(\Omega \times (0,T))^n}
    = \langle \mathcal{G} , \psi \rangle_{\mathcal{D}(\Omega \times (0,T))^n}
    = 0
  $. 
  Since $\mathcal{G}_{\mathrm{reg}}$ belongs to $L^2( 0,T ; H^{-1}( \Omega )^n )$, by Theorem~\ref{thm:reconstruction_pressure_abstract}, there exists a function $p \in L^2( 0,T ; L^2 ( \Omega ) )$ such that, for the distributional gradient $\nabla p$, there holds
  \begin{equation*}
    \langle \nabla p , \psi \rangle_{\mathcal{D}(\Omega \times (0,T))^n}
    = \langle \mathcal{G}_{\mathrm{reg}} , \psi \rangle_{\mathcal{D}(\Omega \times (0,T))^n}
    \ \text{for all} \ \psi \in C^\infty_0( \Omega \times (0,T) )^n.
  \end{equation*}
For any $\psi \in C^\infty_0( \Omega^- )^n$, since $\rho = \beta_1$ in $\Omega^-$ and $v\in L^2( 0,T ; W^{2,2}(\Omega^-(t))^n)$, this leads to
  \begin{equation*}
  \begin{split}
    &\langle \nabla p , \psi \rangle_{\mathcal{D}(\Omega \times (0,T))^n}\\
    ={}&- \int_0^T \!\int_{\Omega^-(t)} \rho ( \partial_t v +  \dive( v \otimes v ) ) \cdot \psi \,\mathrm{d}x \,\mathrm{d}t\\
    &- 2 \int_0^T \!\int_{\Omega^-(t)}\mu( \rho ) D v : D \psi \,\mathrm{d}x \,\mathrm{d}t
    + 2 \stc \int_0^T \!\int_{\Omega^-(t)} K \cdot \psi \,\mathrm{d} x \,\mathrm{d} t\\
    ={}&- \int_0^T \!\int_{\Omega^-(t)} \beta_1 ( \partial_t v +  \dive( v \otimes v ) ) \cdot \psi \,\mathrm{d}x \,\mathrm{d}t\\
    &+ 2 \int_0^T \!\int_{\Omega^-(t)}\mu(\beta_1) \dive( Dv )  \cdot \psi  \,\mathrm{d}x \,\mathrm{d}t
    + 2 \stc \int_0^T \!\int_{\Omega^-(t)} K \cdot \psi \,\mathrm{d} x \,\mathrm{d} t.
  \end{split}
  \end{equation*}
  As $\partial_t v \in L^2( 0,T ; L^2( \Omega^-(t) )^n )$, due to Assumptions~\ref{ass:consistency}, this implies that $p^-$ belongs to $L^2( 0,T ; H^1(\Omega^-(t)))$.
  Additionally, since $v$ is divergence free, we conclude the first identity in \eqref{eq:nabla_p_minus}.
  As the statements about $p^+$ follow analogously, this finishes the proof.
\end{proof}

\subsection{The Pressure Jump}\label{ssec:pressure_jump}

The question left to answer is whether there are pressure functions $p^\pm$ such that the Young--Laplace law \eqref{eq:sim:pressure_jump} holds true.
That is, whether
\begin{equation}\label{eq:yl}
  \left[ p \right] \nu^-
  = ( p^+ - p^- ) \nu^-
  = 2 \big( \mu(\beta_2) (Dv)^+ - \mu(\beta_1) (Dv)^- \big)\nu^- + 2 \stc \kappa \nu^-
\end{equation}
is satisfied on the interface $\Gamma(t)$.
A first step towards an affirmative answer to this question is to understand the "jump brackets" $[ \,\cdot\, ]$ in an appropriate sense: in Theorem~\ref{thm:reconstruction_pressure} we reconstructed a pressure function $p$ such that, for its restrictions $p^\pm$ to $\Omega^\pm$, there holds
\begin{equation*}
  p^+ = \left. p \right|_{\Omega^+} \in L^2( 0,T ; H^1(\Omega^+(t)) )
  \ \text{and} \
  p^- = \left. p \right|_{\Omega^-} \in L^2( 0,T ; H^1(\Omega^-(t)) ).
\end{equation*}
In particular, for a.e.\ $t\in(0,T)$, the traces $\left. p^\pm(t) \right|_{\Gamma(t)}$ are well-defined in the Sobolev sense, and there holds
\begin{equation*}
  \left. p^\pm(t) \right|_{\Gamma(t)} \in H^{\scriptstyle \frac{1}{2}}(\Gamma(t)).
\end{equation*}

Therefore, the statement of Theorem~\ref{thm:reconstruction_pressure} suggests that the pressure jump $\left[ p \right]$ on the space-time interface
\begin{equation}\label{eq:def_spacetime_interface}
  \Gamma = \bigcup_{t\in (0,T)} \big( \Gamma(t) \times \{t\} \big)
\end{equation}
belongs to the space $L^2( 0,T ; H^{\scriptstyle \frac{1}{2}}(\Gamma(t)))$ which is given by either of the equivalent definitions
\begin{equation*}
  L^2( 0,T ; H^{\scriptstyle \frac{1}{2}}(\Gamma(t)))
  = \left\{ \left. u \right|_{\Gamma} \,:\, u \in L^2(0,T; H^1(\Omega^-(t))) \right\}
\end{equation*}
and
\begin{equation*}
  L^2( 0,T ; H^{\scriptstyle \frac{1}{2}}(\Gamma(t)))
  = \left\{ \left. u \right|_{\Gamma} \,:\, u \in L^2(0,T; H^1(\Omega^+(t))) \right\}.
\end{equation*}
Likewise, we introduce the $n$-dimensional version of the latter space and define
\begin{equation*}
  L^2( 0,T ; H^{\scriptstyle \frac{1}{2}}(\Gamma(t))^n) 
  = L^2( 0,T ; H^{\scriptstyle \frac{1}{2}}(\Gamma(t)))^n.
\end{equation*}

%
To give the jump condition~\eqref{eq:yl} a meaning, in the remaining part of this chapter, we will interpret the "jump brackets" $[\,\cdot\,]$ in the sense of Sobolev traces without changing the notation.
More precisely, for a function $f \in L^2( 0,T ; L^2(\Omega) )$ such that the restrictions $f^\pm = \left. f \right|_{\Omega^\pm}$ to $\Omega^\pm$ belong to $L^2( 0,T ; H^1(\Omega^\pm(t)))$, we denote
\begin{equation*}
  \left[ f(t) \right]
  = \left. f^+(t) \right|_{\partial \Omega^+(t) \cap \Omega} - \left. f^-(t) \right|_{\partial \Omega^-(t)},
\end{equation*}
where $\left. f^+(t) \right|_{\partial \Omega^+(t) \cap \Omega}$ and $\left. f^-(t) \right|_{\partial \Omega^-(t)}$ denote the traces of $f^\pm(t)$ on the interface $\Gamma(t)= \partial \Omega^+(t) \cap \Omega = \partial \Omega^-(t)$ taken with respect to the domains $\Omega^+(t)$ and $\Omega^-(t)$, respectively.
Analogously, for a function $f \in L^2( 0,T ; L^2(\Omega)^n )$ such that the restrictions $f^\pm = \left. f \right|_{\Omega^\pm}$ to $\Omega^\pm$ belong to 
$L^2( 0,T ; H^1(\Omega^\pm(t))^n)$, we denote $\left[ f(t) \right] = ( \left[ f_i(t) \right] )_{i=1,\dots,n}$.
To construct a pressure function respecting the Young--Laplace law, we provide the following two technical lemmas.

\begin{lemma}\label{lem:divergence_free_boundary}
  Let $D \subset \Omega$ be a bounded subdomain of $\Omega$ with Lipschitz boundary $\partial D$ and outer normal $\nu_D$.
  Then, for $a \in H^{\scriptstyle \frac{1}{2}}(\partial D )^n$, there holds $\int_{\partial D} a \cdot \nu_D \,\mathrm{d}\Hm^{n-1}(x) = 0$ if and only if there exists a function $u\in H^1(D)^n$ such that
  \begin{equation}\label{eq:div_free_boundary_values} 
    \dive(u) = 0 \ \text{in} \ D
    \ \text{and} \ 
    u = a  \ \text{on} \  \partial D.
  \end{equation}
\end{lemma}

\begin{proof}
  Let $u\in H^1(D)^n$ satisfy \eqref{eq:div_free_boundary_values}.
  Then there holds
   \begin{equation*}
    \int_{\partial D} a \cdot \nu_D \,\mathrm{d}\Hm^{n-1}(x) = \int_{\partial D} u \cdot \nu_D \,\mathrm{d}\Hm^{n-1}(x) = \int_D \dive( u ) \,\mathrm{d} x = 0. 
  \end{equation*}
  The opposite direction follows by~\cite[Theorem IV.1.1]{galdi}.
\end{proof}

The following variant of the fundamental lemma of the calculus of variations allows one to deal with divergence-free test functions.

\begin{lemma}\label{lem:fund_lem_divfree}
  Let Assumptions~\ref{ass:consistency} hold true.
  If $b \in L^2( 0,T ; H^{\scriptstyle \frac{1}{2}}(\Gamma(t))^n)$ satisfies
  \begin{equation*}
    \int_0^T \!\int_{\Gamma(t)} b(t) \cdot \psi(t) \,\mathrm{d}\Hm^{n-1}(x) \,\mathrm{d} t = 0
  \end{equation*}
  for all $\psi \in C^\infty_0( (0,T) ; C^\infty_{0,\sigma}(\Omega))$, then the tangential projection $t \mapsto \mathcal{P}_\tau(b(t)) = b(t) - (b(t) \cdot \nu^-(t)) \nu^-(t)$ vanishes on $\Gamma(t)$, i.e., for a.e.\ $t\in (0,T)$, there holds $\mathcal{P}_\tau(b(t)) = 0$ on $\Gamma(t)$.
  Moreover, the normal projection $t \mapsto \mathcal{P}_{\nu^-}(b(t)) =  (b(t) \cdot \nu^-(t)) \nu^-(t)$ belongs to $L^2( 0,T ; H^{\scriptstyle \frac{1}{2}}(\Gamma(t))^n)$. Moreover, for a.e.\ $t\in (0,T)$, there holds
    \begin{equation}\label{eq:repr_normal_projection}
      \mathcal{P}_{\nu^-}(b(t))  = C ( t ) \nu^-(t)
    \end{equation}
    on $\Gamma(t)$, where the function $t \mapsto  C(t)$ belongs to $L^2(0,T)$ and is given by
    \begin{equation}\label{eq:def_C}
      C(t) = \tfrac{1}{\Hm^{n-1}(\Gamma(t))} \left( \int_{\Gamma(t)} b(t) \cdot \nu^-(t) \,\mathrm{d}\Hm^{n-1}(x) \right).
    \end{equation}
\end{lemma}

\begin{proof}
  We split the proof of the lemma into several steps.
  \setcounter{countproofstep}{0}%
  \proofstep
    By assumption, for a.e.\ $t \in (0,T)$, the fundamental lemma of the calculus of variations implies
    \begin{equation*}
      \int_{\Gamma(t)} b(t) \cdot \psi \,\mathrm{d}\Hm^{n-1}(x) = 0
    \end{equation*}
    for all $\psi \in C_{0,\sigma}^\infty(\Omega)$.
    Passing on to $H^1_{0,\sigma} (\Omega)$, for any $u \in H^1_{0,\sigma} (\Omega)$, there holds
    \begin{equation}\label{eq:b_on_H1}
      \int_{\Gamma(t)} b(t) \cdot u \,\mathrm{d}\Hm^{n-1}(x) = 0.
    \end{equation}
    Let $a \in H^{\scriptstyle \frac{1}{2}}( \Gamma(t) )^n$ be such that $\int_{\Gamma(t)} a \cdot \nu^-(t) \,\mathrm{d}\Hm^{n-1}(x) = 0$.
    Applying Lemma~\ref{lem:divergence_free_boundary} on $\Omega^-(t)$ and $\Omega \setminus \overline{ \Omega^-(t) }$, respectively, there exist functions $u_1 \in H^1(\Omega^-(t))^n$ and $u_2 \in H^1( \Omega \setminus \overline{ \Omega^-(t) })^n$ such that
    $\dive( u_1 ) = 0$ in $\Omega^-(t)$, $\dive( u_2 ) = 0$ in $\Omega \setminus \overline{ \Omega^-(t) }$, $u_1 = u_2 = a \ \text{on} \ \Gamma(t)$ and $u_2 = 0$ on $\partial \Omega$.
    As $H^1_{0,\sigma} (\Omega) = \{ u \in H^1_0(\Omega)^n \,:\, \dive(u) = 0 \}$, see Lemma~\ref{lem:sigma_characterisations}, the composite function
    \begin{equation*}
      u =
      \begin{cases}
        u_1 \ &\text{in} \ \Omega^-(t),\\
        u_2 \ &\text{in} \ \Omega \setminus \overline{ \Omega^-(t) }
      \end{cases}
    \end{equation*}
    is an admissible test function in \eqref{eq:b_on_H1}, which finally yields
    \begin{equation}\label{eq:int_b_a_zero}
      \int_{\Gamma(t)} b(t) \cdot a \,\mathrm{d}\Hm^{n-1}(x)
      = \int_{\Gamma(t)} b(t) \cdot u \,\mathrm{d}\Hm^{n-1}(x)
      = 0.
    \end{equation}

 \proofstep
   For a.e.\ $t\in (0,T)$, there holds $b(t) \in H^{\scriptstyle \frac{1}{2}}( \Gamma(t) )^n$ and $\nu^-(t) \in C^1(\Gamma(t))^n$, by assumption.
%
%
%
Hence, the function $\mathcal{P}_\tau(b(t))$ belongs to $H^{\scriptstyle \frac{1}{2}}( \Gamma(t) )^n$ and satisfies
    \begin{equation*}
      \mathcal{P}_\tau(b(t)) \cdot \nu^-(t)
      = (b(t) - (b(t) \cdot \nu^-(t)) \nu^-(t)) \cdot \nu^-(t)
      = 0.
    \end{equation*}
    In particular,
$\mathcal{P}_\tau(b(t))$ is an admissible choice in \eqref{eq:int_b_a_zero}, which implies
    \begin{equation*}
      \int_{\Gamma(t)} \left| \mathcal{P}_\tau(b(t)) \right|^2 \,\mathrm{d}\Hm^{n-1}(x)
      = \int_{\Gamma(t)} b(t) \cdot \mathcal{P}_\tau(b(t)) \,\mathrm{d}\Hm^{n-1}(x)
      = 0,
    \end{equation*}
    since $\left| \mathcal{P}_\tau(b) \right|^2 = b \cdot \mathcal{P}_\tau(b)$.
    This proves the first claim.
  
 \proofstep
    From $\mathcal{P}_{\tau}(b) = 0$, we infer that $\mathcal{P}_{\nu^-}(b) = b \in L^2( 0,T ; H^{\scriptstyle \frac{1}{2}} ( \Gamma(t))^n )$.
    For the proof of \eqref{eq:repr_normal_projection}, we consider $a(t) \in H^{\scriptstyle \frac{1}{2}}( \Gamma(t) )^n$ given by
    \begin{equation*}
    \begin{split}
      a(t)
      &= \mathcal{P}_{\nu^-}(b(t)) - \frac{1}{\Hm^{n-1}(\Gamma(t))} \left( \int_{\Gamma(t)} b(t) \cdot \nu^-(t) \,\mathrm{d}\Hm^{n-1}(x) \right) \nu^-(t)\\
      &= \left( b(t) \cdot \nu^-(t)  - \tfrac{1}{\Hm^{n-1}(\Gamma(t))}  \int_{\Gamma(t)} b(t)  \cdot \nu^-(t)  \,\mathrm{d}\Hm^{n-1}(x) \right) \nu^-(t).
    \end{split}
    \end{equation*}
    Since, by the definition of $a(t)$, $\int_{\Gamma(t)} a(t)  \cdot \nu^-(t)  \,\mathrm{d}\Hm^{n-1}(x)$ vanishes, $a(t)$ is an admissible function in \eqref{eq:int_b_a_zero}.
    Thus we get
    \begin{equation*}
      \int_{\Gamma(t)} b(t) \cdot a(t) \,\mathrm{d}\Hm^{n-1}(x) = 0
    \end{equation*}
    and, by the definition of $a(t)$, we conclude that
    \begin{equation*}
      \int_{\Gamma(t)} (b(t) \cdot \nu^-(t))^2  \,\mathrm{d}\Hm^{n-1}(x) 
      = \tfrac{1}{\Hm^{n-1}(\Gamma(t))}  \left( \int_{\Gamma(t)} b(t) \cdot \nu^-(t) \,\mathrm{d}\Hm^{n-1}(x) \right)^2
    \end{equation*}
    and hence, we have
%
    \begin{equation*}
    \begin{split}
      &\int_{\Gamma(t)} \left| a(t) \right|^2 \,\mathrm{d}\Hm^{n-1}(x)\\
      ={}&\int_{\Gamma(t)} ( b(t) \cdot \nu^-(t))^2  \,\mathrm{d}\Hm^{n-1}(x) - \tfrac{1}{\Hm^{n-1}(\Gamma(t))}  \left ( \int_{\Gamma(t)} b(t) \cdot \nu^-(t) \,\mathrm{d}\Hm^{n-1}(x) \right)^2\\
      ={}& 0.
    \end{split}
    \end{equation*}
    Hence $a(t)$ vanishes $\Hm^{n-1}$-a.e.\ on $\Gamma(t)$.

 \proofstep
      We shall prove that the function $t \mapsto C(t)$, given by \eqref{eq:def_C}, is measurable.
      To this end, we use the pullback operator $\Phi^-_{-t}\colon L^2( \Gamma ( t ) ) \rightarrow  L^2( \Gamma ( 0 ) )$ introduced in Lemma~\ref{lem:reg_pullback} and define
      $B_i ( t )= \Phi^-_{-t} \circ b_i ( t ) \in L^2( \Gamma ( 0 ) )^n$ for $i=1,2, \dots, n$ and $t\in [0,T]$.
      Then, in view of Lemma~\ref{lem:reg_pullback}, the function $B = (B_1, B_2, \dots, B_n)$ belongs to $L^2( 0,T ; L^2( \Gamma ( 0 ) )^n )$.
      In particular, $B$ is Bochner measurable.
      This means that there exists a sequence $( B_m )_{m\in \N}$ of simple functions $B_m\colon [0,T] \rightarrow L^2( \Gamma ( 0 ) )^n $ such that, for a.e.\ $t\in [0,T]$, there holds $B_m( t ) \rightarrow B( t )$ in $L^2( \Gamma(0) )^n$ as $m\rightarrow \infty$.
      For $t\in [0,T]$ and $A \in L^2( \Gamma ( 0 ) )^n$, define
      \begin{equation*}
        \mathcal I (t,A)
        = \tfrac{1}{\Hm^{n-1}(\Gamma(t))} \left( \int_{\Gamma(t)} ( \Phi^-_t A ) \cdot \nu^-(t) \,\mathrm{d}\Hm^{n-1}(x) \right),
      \end{equation*}
      where $\Phi^-_t$ denotes the inverse of $\Phi^-_{-t}$.
      Using Lemma~\ref{lem:reg_pullback} again, we conclude that, for any $t\in [0,T]$, $\mathcal I (t, \cdot )\colon L^2( \Gamma ( 0 ) )^n \rightarrow \R$ is a linear functional that, for any $A \in L^2( \Gamma ( 0 ) )^n$, satisfies
      \begin{equation*} 
      \begin{split}
       \Hm^{n-1}(\Gamma(t)) \left| \mathcal I (t,A) \right|
        &= \left| \int_{\Gamma(t)} ( \Phi^-_t A ) \cdot \nu^-(t) \,\mathrm{d}\Hm^{n-1}(x) \right|\\
        &\leq \|  \Phi^-_t A \|_{L^2( \Gamma ( t ) )^n} \| \nu^-(t) \|_{L^2( \Gamma ( t ) )^n}
        \end{split} 
      \end{equation*}
      and, consequently,
      \begin{equation}\label{eq:est_func_I} 
        \left| \mathcal I (t,A) \right|
        = {\Hm^{n-1}(\Gamma(t))^{\scriptstyle - \frac{1}{2}}} \|  \Phi^-_t A \|_{L^2( \Gamma ( t ) )^n}
        \leq D \| A \|_{L^2( \Gamma ( 0 ) )^n}
      \end{equation}
      for a constant $D>0$ independent of $t\in [0,T]$.
      Hence $\mathcal I (t) = \mathcal I (t, \cdot )$ defines an element of $\big( L^2( \Gamma ( 0 ) )^n \big)^\ast$, where the constant of continuity does not depend on $t\in [0,T]$.
      Then, $C_m \colon [0,T] \rightarrow \R$, defined by $C_m ( t ) = \mathcal I ( t, B_m ( t ) )$ for $t\in [0,T]$, is a simple function for every $m\in \N$.
      Moreover, for a.e.\ $t\in [0,T]$, we infer that $C_m ( t ) = \mathcal I ( t, B_m ( t ) ) \rightarrow \mathcal I ( t, B ( t ) )$ as $m \rightarrow \infty$. 
      Since there holds 
      \begin{equation}\label{eq:IB_C}
      \begin{split}
        \mathcal I ( t, B ( t ) )
        &= \tfrac{1}{\Hm^{n-1}(\Gamma(t))} \left( \int_{\Gamma(t)} ( \Phi^-_t B ) \cdot \nu^-(t) \,\mathrm{d}\Hm^{n-1}(x) \right)\\
        &= \tfrac{1}{\Hm^{n-1}(\Gamma(t))} \left( \int_{\Gamma(t)} b(t) \cdot \nu^-(t) \,\mathrm{d}\Hm^{n-1}(x) \right)\\
        &= C( t ),
      \end{split}
      \end{equation}
      we conclude that $t \mapsto C( t )$ is a measurable function.

 \proofstep
    In view of \eqref{eq:est_func_I} and \eqref{eq:IB_C}, there is some constant $D>0$ such that
    \begin{equation*}
      \int_0^T \left| C(t) \right|^2 \,\mathrm d t
      = \int_0^T \left| \mathcal  I ( t, B ( t ) ) \right|^2\,\mathrm d t
      \leq D \int_0^T \| B(t) \|_{L^2( \Gamma ( 0 ) )^n}^2\,\mathrm d t.
    \end{equation*}
    Hence we have $\| C \|_{L^2( 0,T )} \leq D \| B \|_{L^2( 0,T ; L^2( \Gamma ( 0 ) )^n)}$ and, as $B \in  L^2( 0,T ; L^2( \Gamma ( 0 ) )^n)$, it follows that $C \in L^2(0,T)$.
  This finishes the proof.
\setcounter{countproofstep}{0}
\end{proof}

The existence statement of Theorem~\ref{thm:reconstruction_pressure} and the preparatory Lemma~\ref{lem:fund_lem_divfree} now allow us to construct a pressure function satisfying the jump condition of the Young--Laplace law.

\begin{theorem}[Reconstruction of pressure]\label{thm:pressure_jump}
  Let Assumptions~\ref{ass:consistency} be satisfied.
  Then there exists a unique function $p \in L^2( 0,T ; L^2( \Omega ) )$ with the following properties.
  \begin{enumerate}
    \item $\left. p \right|_{\Omega^\pm} \in L^2( 0,T ; H^1(\Omega^\pm(t)))$.
    \item For a.e.\ $t\in (0,T)$, there holds $\int_{\Omega} p(t) \,\mathrm{d}x = 0$.
    \item $\nabla p =  - \beta_1 \partial_t v + \mu ( \beta_1 ) \Delta v - \beta_1 (v \cdot \nabla) v$ a.e.\ in $\Omega^-$.
    \item $\nabla p =  - \beta_2 \partial_t v + \mu ( \beta_2 ) \Delta v - \beta_2 (v \cdot \nabla) v$ a.e.\ in $\Omega^+$.
    \item $\left[ p \right] =  2 \left[ \mu( \rho ) D v  \nu^- \right] \cdot \nu^- + 2 \stc \kappa$ in $L^2( 0,T ; H^{\scriptstyle \frac{1}{2}}( \Gamma(t) ) )$.
  \end{enumerate}
\end{theorem}

\begin{proof}
  The uniqueness of $p$ is a direct consequence of the zero-mean condition.
  In the remainder we shall construct the desired function $p$ with the help of the functions $p^\pm$ from Theorem~\ref{thm:reconstruction_pressure} and the function $K= \nabla m \chi_{\Omega^-}$ from Lemma~\ref{lem:extension_mean_curvature}: 
  Consider the function 
  \begin{equation*}
    \tilde p =
    \begin{cases}
      p^- - 2 \stc m&\ \text{in} \ \Omega^-, \\
      p^+ &\ \text{in} \ \Omega^+.
    \end{cases}
  \end{equation*}
  Notice that, by Lemma~\ref{lem:extension_mean_curvature} and Theorem~\ref{thm:reconstruction_pressure}, $\left. \tilde p \right|_{\Omega^\pm}$ belongs to $L^2( 0,T ; H^1(\Omega^\pm(t)))$ and, in the almost-everywhere sense, there holds
  \begin{equation*}
  \begin{split} 
    \nabla \tilde p
    = \nabla p^- - 2 \stc \nabla m
    &=  - \beta_1 \partial_t v + \mu ( \beta_1 ) \Delta v - \beta_1 (v \cdot \nabla) v + 2 \stc K - 2 \stc \nabla m\\
    &=  - \beta_1 \partial_t v + \mu ( \beta_1 ) \Delta v - \beta_1 (v \cdot \nabla) v 
  \end{split}
  \end{equation*}
  in $\Omega^-$ and, likewise, $\nabla \tilde p = \nabla p^+
    =  - \beta_2 \partial_t v + \mu ( \beta_2 ) \Delta v - \beta_2 (v \cdot \nabla) v$ a.e.\ in $\Omega^+$.
  We remark that these properties remain valid for
  \begin{equation*}
    p =
    \begin{cases}
     \left. \tilde p \right|_{\Omega^-} + C^- &\ \text{in} \ \Omega^-, \\
     \left. \tilde p \right|_{\Omega^+} + C^+ &\ \text{in} \ \Omega^+
    \end{cases}
  \end{equation*}
  for arbitrary functions $C^-,C^+ \in L^2(0,T)$.
  Therefore, it is sufficient to prove that there exists some $C \in L^2(0,T)$ such that, for a.e.\ $t\in (0,T)$, there holds
  \begin{equation}\label{eq:jump_ptilde}
    [ \tilde p(t) ]
    = 2 [ \mu(\rho(t)) ( D v(t)  \nu^-(t)) \cdot \nu^-(t) ] -  2 \stc \kappa(t) + C(t)
  \end{equation}
  on $\Gamma(t)$.
  This is because the functions $C^-$ and $C^+$ provide two degrees of freedom: for a.e.\ $t\in (0,T)$, the first may be used to remove the function $C$ from the previous equation.
  For example, by making the choice $C^-= C$ and $C^+ = 0$, the function $p$ satisfies the desired jump condition.
  If $p$ does not have the zero-mean property, the second degree of freedom may be used to subtract its mean value. 
  \par
  Let $\psi \in C_0^\infty( (0,T) ; C_{0,\sigma}^\infty( \Omega ) )$.
  By the definition of $\tilde p$, there holds
  \begin{equation*}
    [ \tilde p ]
    = ( \tilde p )^+ -( \tilde p )^-
    = p^+ -  ( p - 2 \stc m )^-
    = [ p ] + 2 \stc \kappa.
  \end{equation*}
  This implies
  \begin{equation*}
  \begin{split}
    &\int_0^T \!\int_{\Gamma(t)} \left[ \tilde p \right]  \nu^- \cdot \psi \,\mathrm{d}\Hm^{n-1}(x) \,\mathrm{d}t\\
    ={}&- \int_0^T \!\int_{\Omega^-(t)} \nabla p^- \cdot \psi  \,\mathrm{d}x \,\mathrm{d}t
      - \int_0^T \!\int_{\Omega^+(t)} \nabla p^+ \cdot \psi  \,\mathrm{d}x \,\mathrm{d}t\\
     &+ 2 \stc \int_0^T \!\int_{\Gamma(t)} \kappa \nu^- \cdot \psi \,\mathrm{d}\Hm^{n-1}(x) \,\mathrm{d}t.
  \end{split}
  \end{equation*}
  In view of Proposition~\ref{prop:cons:div_free_normal_velocity} and Theorem~\ref{thm:reconstruction_pressure}, we conclude that
  \begin{equation*}
  \begin{split}
  &\int_0^T \!\int_{\Gamma(t)} ( \left[ \tilde p \right] - 2 \stc \kappa ) \nu^- \cdot \psi \,\mathrm{d}\Hm^{n-1}(x) \,\mathrm{d}t\\
    ={}& \int_0^T \!\int_{\Omega \setminus \Gamma(t)} \left( \rho \partial_t v  + \rho \dive (v \otimes v)  - 2 \mu ( \rho ) \dive( Dv )  - 2 \stc K \right) \cdot \psi  \,\mathrm{d}x \,\mathrm{d}t.
  \end{split}
  \end{equation*}
  Recalling that $v \in L^2( 0,T ; H^1( \Omega)^n )$ and applying integration by parts leads to 
  \begin{equation*}
  \begin{split}
    &\int_0^T \!\int_{\Gamma(t)} ( \left[ \tilde p \right] - 2 \stc \kappa ) \nu^- \cdot \psi \,\mathrm{d}\Hm^{n-1}(x) \,\mathrm{d}t\\
    ={}& \int_0^T \!\int_{\Omega \setminus \Gamma(t)} \rho \partial_t v \cdot \psi \,\mathrm{d}x \,\mathrm{d}t\\
    &+ \int_0^T \!\int_{\Omega} \left( \rho \dive (v \otimes v)  + 2 \mu ( \rho ) Dv : D\psi   - 2 \stc K \right) \cdot \psi  \,\mathrm{d}x \,\mathrm{d}t\\
    &+ \int_0^T \!\int_{\Gamma(t)} 2 \left[ \mu( \rho ) D v \right] \nu^-  \cdot \psi \,\mathrm{d}\Hm^{n-1}(x) \,\mathrm{d}t\\
   ={}& - \langle \mathcal{G}_{\mathrm{reg}} , \psi \rangle_{\mathcal{D}(\Omega \times (0,T))^n} + \int_0^T \!\int_{\Gamma(t)} 2 \left[ \mu( \rho ) D v \right] \nu^-  \cdot \psi \,\mathrm{d}\Hm^{n-1}(x) \,\mathrm{d}t,
  \end{split}
  \end{equation*}
  where $\mathcal{G}_{\mathrm{reg}}$ is given by \eqref{eq:def_G_reg}.
  Since $\langle \mathcal{G}_{\mathrm{reg}} , \psi \rangle_{\mathcal{D}(\Omega \times (0,T))^n} = 0$, in view of Proposition~\ref{prop:reg_G_reg}, we conclude that
  \begin{equation*}
    \int_0^T \!\int_{\Gamma(t)} \left( \left[ \tilde p \nu^- - 2 \mu( \rho ) D v  \nu^- \right] - 2 \stc \kappa \nu^- \right)\cdot \psi \,\mathrm{d}\Hm^{n-1}(x) \,\mathrm{d}t = 0.
  \end{equation*}
  
%
  Finally, in view of Lemma~\ref{lem:fund_lem_divfree}, there exists a function $C\in L^2( 0 ,T )$ such that \eqref{eq:jump_ptilde} is valid.
\end{proof}



\bibliographystyle{abbrv}

\end{document}